\documentclass{amsart}
\usepackage{graphicx}
\usepackage{amssymb}
\usepackage{amsfonts}
\swapnumbers
\newtheorem{theorem}{Theorem}[section]
\newtheorem{lemma}[theorem]{Lemma}

\newtheorem{proposition}[theorem]{Proposition}
\theoremstyle{definition}

\newtheorem{assumption}[theorem]{Assumption}
\newtheorem{remark}[theorem]{Remark}

\numberwithin{equation}{section}
 \theoremstyle{plain}
 
 \numberwithin{equation}{section} 
 \numberwithin{figure}{section} 
 \theoremstyle{plain}
 \theoremstyle{plain}
 \theoremstyle{remark}
 \newtheorem*{acknowledgement*}{Acknowledgement}

\newcommand{\cF}{{\mathcal F}}
\newcommand{\cG}{{\mathcal G}}
\newcommand{\cH}{{\mathcal H}}

\newcommand{\cL}{{\mathcal L}}

\newcommand{\cN}{{\mathcal N}}

\newcommand{\te}{{\theta}}

\newcommand{\Om}{{\Omega}}
\newcommand{\om}{{\omega}}
\newcommand{\ve}{{\varepsilon}}
\newcommand{\del}{{\delta}}

\newcommand{\gam}{{\gamma}}
\newcommand{\Gam}{{\Gamma}}
\newcommand{\vf}{{\varphi}}

\newcommand{\Sig}{{\Sigma}}
\newcommand{\sig}{{\sigma}}
\newcommand{\al}{{\alpha}}
\newcommand{\be}{{\beta}}
\newcommand{\ka}{{\kappa}}
\newcommand{\la}{{\lambda}}

\newcommand{\vp}{{\varpi}}

\newcommand{\bbN}{{\mathbb N}}

\newcommand{\bbR}{{\mathbb R}}

\newcommand{\bbZ}{{\mathbb Z}}
\newcommand{\bbI}{{\mathbb I}}

\newcommand{\brF}{{\bar F}}

\begin{document}
\title[]{Berry-Esseen type estimates for nonconventional sums}%
 \vskip 0.1cm
 \author{Yeor Hafouta and Yuri Kifer\\
\vskip 0.1cm
 Institute  of Mathematics\\
Hebrew University\\
Jerusalem, Israel}%
\address{
Institute of Mathematics, The Hebrew University, Jerusalem 91904, Israel}
\email{yeor.hafouta@mail.huji.ac.il, kifer@math.huji.ac.il}%

\thanks{ }
\subjclass[2000]{Primary: 60F05 Secondary: 60J05}%
\keywords{central limit theorem, Berry-Esseen theorem, mixing, nonconventional
 setup.}%
\dedicatory{  }
 \date{\today}
\begin{abstract}\noindent
We obtain Berry-Esseen type estimates for "nonconventional" expressions
of the form $\xi_N=\frac{1}{\sqrt{N}}
\sum_{n=1}^N(F(X(q_{1}(n)),...,X(q_{\ell}(n)))-\brF)$
where $X(n)$ is a sufficiently fast mixing vector process with some
moment conditions and stationarity properties, $F$ is a continuous
function with polynomial growth and certain regularity properties,
$\brF=\int Fd(\mu\times...\times\mu)$, $\mu$ is the distribution of
$X(0)$ and $q_{i}(n)=in$ for $1\leq i\leq k$ while
for $i>k$ they are positive functions taking integer values on integers
with some growth conditions which are satisfies, for instance, when
they are polynomials of increasing degrees. Our setup is similar to \cite{KV1}
where a nonconventional functional central limit theorem was obtained and
the present paper provides estimates for the convergence speed.
As a part of the study we provide answers for the crucial question on
positivity of the limiting variance $\lim_{N\to\infty}$Var$(\xi_N)$ which
was not studied in \cite{KV1}. Extensions to the continuous time case
 will be
 discussed as well. As in \cite{KV1} our results are applicable
 to stationary processes generated by some classes of sufficiently well
 mixing Markov chains and dynamical systems.
 \end{abstract}
\maketitle
\markboth{Y. Hafouta and Y. Kifer }{Berry Esseen Theorem}
\renewcommand{\theequation}{\arabic{section}.\arabic{equation}}
\pagenumbering{arabic}

\renewcommand{\theequation}{\arabic{section}.\arabic{equation}}
\pagenumbering{arabic}

\section{Introduction}\label{sec1}\setcounter{equation}{0}

The classical Berry-Esseen theorem provides a uniform estimate of the error
term in the central limit theorem for a sum of mean zero independent
identically distributed (i.i.d.) random variables $\{ X(n)\}^\infty_{n=1}$.
Namely, let $F_n$ be the distribution function of $\frac{1}{\sigma\sqrt{n}}
\sum_{i=1}^{n}X(i)$ where $\sigma=\sqrt {E(X(1))^{2}}>0$ and $\Phi$ be the
standard normal distribution function then
\begin{equation}\label{1.1}
\sup_{x\in\bbR}|F_{n}(x)-\Phi(x)|\leq\frac{CE|X(1)|^{3}}{\sigma^{3}\sqrt{n}}
\end{equation}
(see \S 6 of Ch. III in \cite{Shr}) where $C$ is an absolute constant which by
efforts of many researchers was optimized by now to a number a bit less
than $1/2$.

Motivated partially by the research on nonconventional ergodic theorems
(the name comes from \cite{Fur}) the study of nonconventional limit theorems
was initiated in \cite{Ki1}. More recently a functional central limit theorem
was proved in \cite{KV1} for normalized nonconventional sums of the form
\begin{equation}\label{1.2}
\xi_{N}(t)=\frac{1}{\sqrt{N}}\sum_{Nt\geq n\geq1}\big(F(X(q_{1}(n)),...,
X(q_{\ell}(n)))-\brF\big)
\end{equation}
where $\{X(n),  n\geq 0\}$ is a sufficiently fast mixing vector valued
process with some stationarity properties satisfying certain moment
conditions, $F$ is a continuous function with polynomial growth rate
and certain regularity properties, $\bar{F}=\int Fd(\mu\times...\times\mu)$,
$\mu$ is the common distribution of $X(n)$'s and $q_{j}(n)=jn$ for
$1\leq j\leq k$ while $q_{j}(n)$ for $k<j\leq\ell$ are positive functions
taking integer values on integers and satisfying certain growth conditions.
In this paper we derive Berry-Esseen type estimates for the convergence rate
in such nonconventional limit theorems.

During last 50 years central limit theorems were extended to weakly dependent
sequences of random variables and to martingale differences and corresponding
Berry-Esseen type estimates of the speed of convergence were obtained, as well
(see, for instance, \cite{HH}, \cite{Ri}, \cite{RR}, \cite{Cou}
 and references there). We observe though that summands in nonconventional sums
 appearing in (\ref{1.2}) are usually strongly long range dependent (even when
 $X(n),n\geq 1$ are independent) so the
 results for the weakly dependent case are not applicable here. Still, it was
 shown in \cite{KV1} that under natural conditions nonconventional sums can
 be splitted into $\ell$ subsums and each of the latter can be approximated by
 a martingale. We will show that, actually, in the arithmetic progression case
  $q_j(n)=jn,\, j=1,...,\ell$ the whole nonconventional sum can be approximated
  by one martingale which will enable us to apply one of Berry-Esseen type
  results for martingales mentioned above. Still, in order to do so
  we will need to obtain appropriate asymptotic covariance estimates for
  nonconventional summands.
 We observe that when not all $q_j(n)$'s are linear but, say, $q_j(n),\,
 j=k+1,...,\ell$ grow faster, for instance, polynomially as in \cite{KV1}
 then we have to deal with several martingales with respect to different
 filtrations which requires additional considerations described in the
 concluding Section \ref{sec6}.

 As (\ref{1.1}) and more advanced results show Berry-Esseen type estimates
 (with an absolute constant) depend crucially on variances of the corresponding
 sums which appear in some form in the denominators of corresponding
 bounds. In the standard (conventional) setup the conditions which ensure
 linear growth in the number of summands of these variances are well known
 for stationary sequences since \cite{IL}. On the other hand, the limiting
 behavior of the variance $\xi_N$ in (\ref{1.2}) was not studied in \cite{KV1}
 in spite of the fact that a meaningful central limit theorem requires the
 limit of Var$\xi_N$ as $N\to\infty$ to be positive. Some partial results in
 this direction were obtained in \cite{Ki1} and \cite{HK}. Ensuring
 positivity of the limiting variance and obtaining appropriate lower bounds
 for it is especially important in Berry-Esseen type estimates and we provide
 here a rather complete answer concerning this question. Namely, we show
 that under appropriate mixing conditions the positivity question for the
 limiting variance of $\xi_N$ can be reduced to the same question for the
 $\ell$-dimensional process constructed of independent copies of the process
 $X(n),\, n\geq 0$ which is pluged in the function $F$. If $X(n),\, n\geq 0$
 is stationary then (in the $k=\ell$ case) this $\ell$-dimensional process
  is stationary, as well,
 and we can rely on the well known results concerning the latter (see
 \cite{IL} and the next section).

 The structure of this paper is the following. In the next Section \ref{sec2}
 we describe precisely our setup and formulate our main results. In Section
 \ref{sec3} we derive some auxiliary estimates. In Sections \ref{sec4} and
 \ref{sec5} we prove our main theorems. In order to increase readability of
 the paper we do not treat the most general case in the main part and postpone
 extensions and generalizations till the concluding Section \ref{sec6}. There
 we consider also the case of independent $X(n)$'s where without relying on
 martingale results but employing a more direct method we are able to provide
 substantially better estimates than in the general situation.

\section{Preliminaries and main results}\label{sec2}\setcounter{equation}{0}

\subsection{Setup and assumptions}

Our setup consists of a $\wp$-dimensional stochastic
process $\{X(n)\}_{n\geq0}$ on a probability space $(\Om,\cF,P)$
and a nested family of $\sigma-algebras$ $\cF_{k,l}$, $-\infty\leq k\leq l
\leq\infty$
such that $\cF_{k,l}\subset\cF_{k',l'}$ if $k'\leq k$
and $l'\geq l$. As usual (see \cite{Br}) the dependence between two sub
$\sigma-algebras$ $\cG,\cH\subset\cF$ will be measured by the expressions
\begin{equation}\label{2.1}
\varpi_{q,p}(\cG,\cH)=\sup\{||E(g|\cG)-Eg||_{p}\,:\thinspace g\in L^{q}
(\Omega,\cH,P) \mbox{ and  }  ||g||_{q}\leq1\}
\end{equation}
and we refer the reader to \cite{Br} for relations between various dependence
coefficients. Set also
\begin{equation}\label{2.2}
\varpi_{q,p}(n)=\sup_{k\geq 0}\varpi_{q,p}(\cF_{-\infty,k},\cF_{k+n,\infty}).
\end{equation}

Our results below can be obtained assuming that $X(n)$ is measurable with
respect to $\cF_{n,n}$ without special assumptions on the function $F$
beyond measurability similarly to \cite{HK}. Nevertheless, we prefer here
the setup from \cite{KV1} which allows applications to dynamical systems.
 Thus, we introduce approximation rate coefficients
\begin{equation}\label{2.3}
\be(q,r)=\sup_{k\geq0}||X(k)-E(X(k)|\cF_{k-r,k+r})||_{q}.
\end{equation}
We will not require stationarity of the process $\{X(n), n>0\}$
assuming only that the distribution of $X(n)$ does not depend on
$n$ and the joint distribution of $(X(n),X(n'))$ depends only on
$n-n'$ which we write for further reference by
\begin{equation}\label{2.4}
X(n)\thicksim\mu  \mbox{ and  } (X(n),X(n'))\thicksim\mu_{n-n'}
\end{equation}
where $Y\thicksim\mu$ means that $Y$ has $\mu$ for it's distribution,
denoted also $\mu=\cL(Y)$.

Next, let $F=F(x_{1},...,x_{\ell})$, $x_{j}\in\bbR^{\wp}$
be a function on $\bbR^{\wp\ell}$ such that for some $K,\iota>0$,
$\ka\in(0.1]$ and all $x_{i},y_{i}\in\bbR^{\wp}$, $i=1,...,\ell$,
\begin{eqnarray}\label{2.5}
&|F(x)-F(y)|\leq K(1+\sum_{i=1}^{\ell}(|x_{i}|^{\iota}+|y_{i}|^{\iota}))
\sum_{i=1}^{\ell}|x_{j}-y_{j}|^{\ka}\\
&\mbox{and}\,\,\, |F(x)|\leq K(1+\sum_{i=1}^{\ell}|x_{i}|^{\iota})\nonumber
\end{eqnarray}
where $x=(x_{1},...,x_{\ell}),  y=(y_{1},...,y_{\ell})$.
To simplify the formulas we assume a centering condition
\begin{equation}\label{2.6}
\brF=\int F(x_1,...,x_\ell)d\mu(x_{1})...d\mu(x_{\ell})=0
\end{equation}
which is not really a restriction since we can always replace $F$
by $F-\brF.$ Our main goal is obtaining Berry-Esseen and covariance
type estimates for $\xi_{N}(t),\thinspace t\in[0,T]$ defined
in (\ref{1.2}) (with $\brF=0$). For each $\te>0$, set
\begin{equation}\label{2.7}
\gam_{\te}^{\te}=||X(n)||_{\te}^{\te}=\int|x|^{\te}d\mu.
\end{equation}
Our results rely on the following assumptions (similar to \cite{KV1}).
\begin{assumption}\label{ass2.1}
With $d=(\ell-1)\wp$ there exits $\infty>p,q\geq1$, $b\geq2$ ,
$\al,\la\geq0$ and $\del,m>0$ (these numbers will be called the
initial parameters) with $\del<\ka-\frac{d}{p}$ satisfying
\begin{eqnarray}\label{2.8}
\te(q,p,\al,1)=\sum_{n\geq1}n^{\al}\varpi_{q,p}(n)<\infty,\\
\Lambda(q,\del,\la,1)=\sum_{r=1}^{\infty}r^{\la}(\be(q,r))^{\del}
<\infty,\label{2.9} \\
\gam_{m}<\infty,\gam_{bq\iota}<\infty; \mbox{ with }\frac{1}{b}
\geq\frac{1}{p}+\frac{\iota+2}{m}+\frac{\del}{q}, \label{2.10}
\end{eqnarray}
while conditions on $\al$ and $\la$ will be specified in the statements
below.
\end{assumption}

As in \cite{KV1} it will be useful to represent the function
$F=F(x_{1},...,x_{\ell})$ in the form
\begin{equation}\label{2.11}
F=F_{1}(x_{1})+...+F_{\ell}(x_{1},...,x_{\ell})
\end{equation}
where for $i<\ell$,
 \begin{equation}\label{2.12}
 F_i(x_1,...,x_i)=\int F(x_1,...,x_\ell)d\mu(x_{i+1})...d\mu(x_\ell)-
 \int F(x_1,...,x_\ell)d\mu(x_i)...d\mu(x_\ell)
 \end{equation}
and
\begin{equation}\label{2.13}
F_{\ell}(x_{1},...,x_{\ell})=F(x_{1},...,x_{\ell})-\int F(x_{1},...,
x_{\ell})d\mu(x_{\ell})
\end{equation}
which ensures that
\begin{equation}\label{2.14}
\int F_{i}(x_{1},...,x_{i-1},x_{i})d\mu(x_{i})=0\mbox{ } \forall
x_{1},...,x_{i-1}.
\end{equation}

Next, assume that $q_j(n)=jn$ for $j=1,...,k\leq\ell$ while when
$\ell\geq j>k$ we have $q_j(n+1)-q_j(n)\to\infty$ and $q_j(\ve n)-q_{j-1}(n)
\to\infty$ as $n\to\infty$ for each $\ve>0$. Following \cite{KV1} we will
use the representation
\begin{equation}\label{2.15}
\xi_{N}(t)=\sum_{i=1}^{k}\xi_{i,N}(it)+\sum_{i=k+1}^{\ell}\xi_{i,N}(t)
\end{equation}
where for $1\leq i\leq k$,
\begin{equation}\label{2.16}
\xi_{i,N}(t)=\frac{1}{\sqrt{N}}\sum_{n=1}^{[\frac{Nt}{i}]}F_{i}\left(X(n),...
,X(in)\right)
\end{equation}
and for $i>k$
\begin{equation}
\xi_{i,N}(t)=\frac{1}{\sqrt{N}}\sum_{n=1}^{[Nt]}F_{i}(X(q_{1}(n)),...,
X(q_{i}(n)))\label{2.17}
\end{equation}
 The following result was proved in \cite{KV1}.

\begin{theorem}\label{thm2.2} Suppose that Assumption  \ref{ass2.1} holds
true with $b=2$ and $\al=\la=0$. Then the $\ell-$dimensional process
$\{\xi_{i,N}(t)\}_{i=1}^{\ell}$
converges in distribution as $N\to\infty$ to a vector Gaussian process
$\{\eta_{i}(t)\}_{i=1}^{\ell}$ with stationary independent mean zero
increments and covariances
\[
E(\eta_{i}(s)\eta_{j}(t))=\min(s,t)D_{i,j}=\lim_{N\to\infty}E(\xi_{i,N}(s)
\xi_{j,N}(t))
\]
where $D_{i,j}=0$ if $i\ne j$ and $\max(i,j)>k$. This together with
(\ref{2.15}) yields that the limiting variance exists and has the form
\[
\lim_{N\to\infty}\mbox{Var}\xi_N(t)=\lim_{N\to\infty}E\xi_N^2(t)=t\sig^2=
t(\sig^2_0+\sig^2_1)
\]
where
\[
\sig^2_0=\lim_{N\to\infty}E(\sum_{i=1}^k\xi_{i,N}(i))^2=
\sum^k_{i=1}iD_{i,i}+2\sum_{1\leq i<j\leq k}iD_{i,j}
\]
\[
\mbox{and}\,\,\quad\sig_1^2=\lim_{N\to\infty}E(\sum_{i=k+1}^\ell\xi_{i,N}(1))^2
=\sum^\ell_{i=k+1}D_{i,i}.
\]
Moreover, the process $\xi_{N}(\cdot)$ converges in distribution
to the Gaussian process $\eta(\cdot)$ which can be represented
in the form $\eta(t)=\sum_{i=1}^{k}\eta_{i}(it)+\sum_{i=k+1}^{\ell}\eta_{i}(t)$
which may have dependent increments.
\end{theorem}

\subsection{Statement of main results.}

In order to make this paper more readable we will focus on the case
$k=\ell$ and introduce several extensions in Section 6 (among them,
results for $k<\ell$). In general, uniform Berry-Esseen type estimates
can only be meaningful if the asymptotical variance $\sig^2=\lim_{N\to\infty}
E\xi^2_N(1)$ is positive which can be seen already in (\ref{1.1}). Some
conditions for positivity of $\sig^2$
were obtained in Theorem 2.3 from \cite{HK} but the following theorem provides
a substantially stronger and more general result.

\begin{theorem}\label{thm2.3}
Suppose that $k=\ell$ and that Assumption \ref{ass2.1} holds true with $b=2$
and $\al,\la\geq1$. Let $\{X^{(i)}(n)\}_{n\geq0}\  i=1,...,\ell$
be $\ell$ independent copies of the process $\{X(n)\}_{n\geq0}$
and set
\[
Z_{n}=F(X^{(1)}(n),X^{(2)}(2n),...,X^{(\ell)}(\ell n)) \mbox{ and }
\Sigma_{N}=\sum_{n=1}^{N}Z_{n}.
\]
Then the limit
\[
s^{2}=\lim_{n\to\infty}\frac{1}{n}\mbox{Var}\Sigma_{n}
\]
exists. Moreover, $\sig^2>0$ if and only if $s^{2}>0$ and the
latter conditions holds true if and only if  there exists no representation
of the form
\[
Z_{n}=V_{n+1}-V_{n},  n=0,1,2...
\]
where $\{V_{n}\}_{n=1}^{\infty}$ is a square integrable weakly (i.e. in the
wide sense) stationary process. Furthermore, $s^2=0$ if and only if Var$\Sig_N$
is bounded in $N$ and then for all $N\geq 2$,
\[
\mbox{Var}\xi_N\leq CN^{-1}\ln^2 N
\]
for some $C>0$ independent of $N$.
\end{theorem}

We observe that this theorem remains true with essentially the same
proof also in the more general case $k<\ell$ described above.
Actually, in this case $\sig^2>0$ unless $F_j=0$ for all $j=k+1,k+2,...,\ell$
$\mu\times\cdots\times\mu$-almost surely (a.s.) (see Section \ref{sec6.2}).
The above result reduces the problem on positivity of the limiting variance
for nonconventional sums to the corresponding much more studied question
concerning sums of stationary in the wide sense processes. If $X(n),\, n\geq 0$
is a strictly stationary process then $(X^{(1)}(n),X^{(2)}(2n),...,
X^{(\ell)}(\ell n))_{n\geq 0}$ and $F(X^{(1)}(n),X^{(2)}(2n),...,
X^{(\ell)}(\ell n)),\, n\geq 0$ are strictly stationary, as well, while under
our condition (\ref{2.4}) these processes are stationary in the wide sense.
Limit theorems for sums of the latter were widely studied.
We observe that it is not possible to give useful (i.e. computable) positive
lower bounds for the limiting variance even in a general conventional situation
of sums of stationary processes. In the nonconventional case the situation is
more complicated and though some formulas for the limiting variances are
given in \cite{KV1} it is not possible to check directly when they are positive.
Still, assuming that $X(n),\, n\geq 0$ are independent we provide in
Section \ref{sec6} some formulas for limiting variances which are easier to
handle and to obtain estimates.

\begin{remark}\label{rem2.3+}
Similarly to \cite{KV1} the results of this paper can be applied to some
types of discrete time dynamical systems $T:\Om\circlearrowleft$ such as
subshifts of finite type, expanding transformations and Axiom A diffeomorphisms
considered with a Gibbs invariant measure $\mu$ (see, for instance, \cite{Bow}).
 Such dynamical systems are
exponentially fast $\psi$-mixing which is more than enough for our purposes. In
this
setup we should take $X(n)=f\circ T^n$ where, say, $f$ is a H\" older continuous
(vector) function. Then Theorem \ref{thm2.3} reduces the question on positivity
of the limiting variance of $N^{-1/2}\sum_{n=1}^NG(T^n\om,T^{2n}\om,...,
T^{\ell n}\om)$, where $G(\om_1,...,\om_\ell)=F(f(\om_1),...,f(\om_\ell))$,
 to the corresponding question for the product dynamical
system $T\times T^2\times\cdots\times T^\ell:\,\Om\times\cdots\times\Om
\circlearrowleft$, i.e.
for $N^{-1/2}\sum_{n=1}^NG(T^n\om_1,T^{2n}\om_2,...,T^{\ell n}\om_\ell)$.
Since $T\times T^2\times\cdots\times T^\ell$ preserves the product measure
$\mu\times\cdots
\times\mu$ and also turns out to be an exponentially fast $\psi$-mixing
dynamical system we arrive at a well studied problem. Furthermore, it is known
since \cite{Bro} that for a general measure preserving dynamical system
$T:\Om\circlearrowleft$ and a bounded measurable function $H$ the sums
$\sum_{n=1}^NH(T^n\om)$ are almost surely uniformly bounded if and only if
$H$ has a co-boundary representation
$H(\om)=\vf(T\om)-\vf(\om)$ for some other bounded measurable function $\vf$.
For nonconventional sums $\sum_{n=1}^NG(T^n\om,...,T^{\ell n}\om)$ such result
cannot hold true in this generality since the meaningful action here is only
on the diagonal of $\Om\times\cdots\times\Om$ and its images under
$T\times T^2\times\cdots\times T^\ell$, and so we can define $G$ to be a
co-boundary for $T\times T^2\times\cdots\times T^\ell$ on the diagonal and
its images which has zero product measure while defining $G$ arbitrarily
outside of the diagonal still preserving measurability. Then the sum will be
bounded but $G$ will not have necessarily a co-boundary representation on the
whole product space. Such simple counterexample will usually be impossible if
 we impose some regularity conditions on $G$, even just continuity.
In the more restricted
nonconventional situation of Theorems \ref{thm2.2} and \ref{thm2.3} the central
limit theorem together with positivity of the limiting variance ensures that
the sum  $\sum_{n=1}^NG(T^n\om,...,T^{\ell n}\om)$ is unbounded while if it is
bounded then $G$ must have a co-boundary representation.  It would still be
interesting to understand whether boundedness of these sums in the
nonconventional setup
can be characterized in a more general situation. In clarifying some points
discussed in this remark the second author benefited from several conversations
with A. Katok at PennState University in September 2014.
\end{remark}

 Recall, that the Kolmogorov (uniform) metric is defined
 for each pair of distribution functions $F,G$ by
\begin{equation}\label{2.18}
d_{K}(F,G)=\sup_{x\in\bbR}|F(x)-G(x)|.
\end{equation}
Now we can formulate our second main result.

\begin{theorem}\label{thm2.4}
Suppose that $k=\ell$ and that Assumption \ref{ass2.1} holds true with $b\geq4$,
 $\al,\la>1$ and that $\sig^2>0$. Then,
\[
d_{K}(\cL(\xi_{N}(1)),\cN(0,\sig^{2}))\leq CA(\sig)N^{-\zeta(\al,\la)}
\]
where, $\cN(0,\sig^2)$ is the zero mean normal distribution with the
variance $\sig^2>0$, the constant $C>0$ depends
 only on the initial parameters and the expressions
(\ref{2.8}) and (\ref{2.9}),
$A(\sig)=(1+\frac{1}{\sig})\max(\sig^{-\frac{4}{3}},\sig^{-\frac{4}{5}})$ and
\[
\zeta(\al,\la)=\frac{1}{10}\min(\min(\al,\la)-1,\frac{\la}{\la+8}).
\]
Moreover, if there exist $c\in(0,1)$ and $r>0$ satisfying $\varpi_{q,p}(n)
+\be(q,n)\leq rc^{n}$
then $N^{-\zeta(\al,\la)}$ can be replaced by  $N^{-\frac{1}{10}}\ln\, N$.
\end{theorem}

In order to describe our method of the proof of Theorem \ref{thm2.4} consider
the simpler case when $X(n),\, n\geq 0$ is a sequence of independent
identically distributed (i.i.d.) random
variables and choose the $\sig$-algebras $\cF_{n,m}=\sig\{ X(n),...,X(m)\}$
for any $n\leq m$. For each $i=1,2,...,\ell$ define
\[
M_{i,n}=\sum_{im\leq n}F_i(X(m),X(2m),...,X(im))\quad\mbox{for}\quad
n\leq iN
\]
and $M_{i,n}=M_{i,iN}$ for $n\geq iN$. Then $M_{i,n},\, n=1,...,\ell N$ is
a martingale with respect to the filtration $\{\cF_{0,n},\, n\geq 0\}$, and
so $M_n=\sum_{i=1}^\ell M_{i,n},\, n=1,2,...,\ell N$ is also a martingale or,
more precisely, a martingale array since the construction depends on $N$.
Now observe that $\xi_N(1)= N^{-1/2}M_{\ell N}$
and Theorem \ref{thm2.4} will follow in this situation from estimates
of rates of convergence in the martingale central
limit theorem derived in \cite{HH}. Still, for this specific i.i.d. case we
will give in
Section \ref{sec6} another more direct proof which yields better estimates.
In the more general setup of the present paper we will need
first a truncation procedure and then a martingale approximation similar
but still somewhat different from \cite{KV1}. Namely, as above in the
i.i.d. case, we construct in the case $k=\ell$ a martingale approximation
 of the whole sum
$\sqrt N\xi_N(1)$ and not only of its parts $\sqrt N\xi_{i,N}(t)$ as in
\cite{KV1}. Some additional work, described in Section \ref{sec6}, is
 needed when $\xi_N$ has the more general form (\ref{1.2}) with some of
$q_j(n)$'s growing faster than linearly. In order to rely on \cite{HH}
we will need also appropriate quadratic variation estimates which will be
obtained in Section \ref{sec5}.

\begin{remark}\label{rem2.4+}
We construct a martingale array approximation (representation in the i.i.d.
case described above) for the whole normalized sum $\xi_N(1)$ and not only
for its parts $\xi_{i,N}$ as in \cite{KV1}. This serves us well for the
Berry-Esseen type estimates here and yields also the central limit theorem
for $\xi_N(1)$ from standard results for martingale arrays. Still,
the functional central limit theorem for the whole process $\xi_N(t),\,
t\geq 0$ cannot be obtained this way. Indeed, if we could approximate this
process by a martingale array depending only on $N$ but not on $t$ then the
limiting Gaussian process would have independent increments which is not the
case in general (see \cite{KV1}). Already in the above construction for the
i.i.d. case we would have to define $M_{i,n}=M_{i,[iNt]}$ for $n\geq [iNt]$
obtaining martingales depending on $t$ which would not enable us to employ
standard theorems on martingale arrays.
\end{remark}

\section{Auxiliary estimates}\label{sec3}
\setcounter{equation}{0}

We start with the following simple observation.
\begin{lemma}\label{lem3.1}
Let $f:(\bbR^{\nu})^{d}\to\bbR$ and $g:(\bbR^{\nu})^{p}\to\bbR$
satisfy the conditions (\ref{2.5}) and (\ref{2.6}). Then the function
$h:(\bbR^{\nu})^{d+p}\to\bbR$ defined by $h(x,y)=f(x)g(y)$
satisfies these conditions with constants $2\iota,\ka$ and $\tilde{K}
=2(1+d+p)K^{2}$ in place of $\iota,\ka$ and $K$, respectively.
\end{lemma}
\begin{proof}
The lemma follows from three simple inequalities $|ab|\leq\frac{1}{2}(a^{2}
+b^{2})$, $|a|\leq1+a^{2}$, $|ab-a'b'|\leq|a(b-b')|+|b'(a-a')|$,
the H\" older continuity of $f$ and $g$ and the concavity of the function
$x\to x^{a}$ for $1>a>0$.
\end{proof}

Next, we will need
\begin{lemma}\label{lem3.2}
Let $0<\del<\ka\leq1$ and $b\geq1$ satisfy $\frac{1}{b}\geq\frac{1}{p}
+\frac{\iota+2}{m}+\frac{\del}{q}$
for some $q,p\geq1$ and $m,\iota>0$. Then for any random variables
$Y,X$,
\[
||Y^{\iota}\cdot X^{\ka}||_{b}\leq(1+||X||_{m}^{\ka})||Y||_{m}^{\iota}
\cdot||X||_{q}^{\del}.
\]
\end{lemma}

\begin{proof}
First, clearly,
\begin{eqnarray}\label{3.1}
&\| Y^{\iota}X^{\ka}\|_{b}\leq T_1+T_2\,\,\mbox{where}\\
& T_1=\| Y^{\iota}X^{\ka}\bbI_{\{|X|>1\}}\|_{b}\,\,\mbox{and}\,\,
T_2=\| Y^{\iota}X^{\ka}\bbI_{\{|X|\leq1\}}\|_{b}.\nonumber
\end{eqnarray}
Observe that $T_{1}=||Y^{\iota}X^{\ka}\cdot(\bbI_{\{|X|>1\}})^{\del}||_{b}$.
Since $\frac{1}{b}>\frac{\iota+\ka}{m}+\frac{\del}{q}$ then Lemma 3.1 from
 \cite{KV1} yields that
\[
T_{1}\leq||Y^{\iota}X^{\ka}||_{\frac{m}{\iota+\ka}}\cdot||\bbI_{\{|X|>1\}}
||_{q}^{\del}.
\]
Since $||\bbI_{\{|X|>1\}}||_{q}=\left(P\{|X|>1\}\right)^{\frac{1}{q}}=
\left(P\{|X|^{q}>1\}\right)^{\frac{1}{q}}$
it follows by the Markov inequality that $||\bbI_{\{|X|>1\}}||_{q}\leq
\left(E|X|^{q}\right)^{\frac{1}{q}}=||X||_{q}$.
Moreover, since $(\frac{m}{\iota+\ka})^{-1}=\frac{\iota}{m}+\frac{\ka}{m}$,
 Lemma 3.1 from \cite{KV1}  yields $||Y^{\iota}X^{\ka}||_{\frac{m}
{\iota+\ka}}\leq||Y||_{m}^{\iota}||X||_{m}^{\ka}$,
and so $T_{1}\leq||Y||_{m}^{\iota}||X||_{m}^{\ka}||X||_{q}^{\del}.$
Next, set $Z=|X|\bbI_{\{|X|\leq1\}}$. Clearly, $T_{2}=||Y^{\iota}
Z^{\ka}||_{b}$.
Since $0\leq Z\leq1$ and $\del<\ka$ it follows that $T_{2}\leq||Y^{\iota}
Z^{\del}||_{b}$.
Since $\frac{1}{b}>\frac{\iota}{m}+\frac{\del}{q}$ we apply again Lemma
3.1 from \cite{KV1} and use that $||Z||_{q}\leq||X||_{q}$
in order to obtain $T_{2}\leq||Y||_{m}^{\iota}||X||_{q}^{\del}$.
The lemma now follows from (\ref{3.1}) and the above estimates.
\end{proof}

We will use also
 \begin{lemma}\label{lem3.3}
Let $X,Y$ and $Z$ be random variables and $\del>0.$ Suppose that
$X$ and $Y$ are defined on a common probability space and $Z$ has
density bounded by $c>0$. Then, for any $a\geq1$,
\[
d_{K}(Y,Z)\leq3d_{K}(X,Z)+||X-Y||_{a}^{\frac{a}{1+a}}(1+4c).
\]
\end{lemma}

\begin{proof}
Let $a,t\in\bbR$ and $\del>0$. Then,
\begin{eqnarray}
&|P\{Y\leq t\}-P\{Z\leq t\}|\leq d_{K}(X,Z)+|P\{X\leq t\}-P\{Y\leq t\}|\\
&\leq d_{K}(X,Z)+P\{|X-t|\leq\del\}+P\{|X-Y|>\del\}.\nonumber
\end{eqnarray}
By the definition of $d_{k}(X,Z)$ and the mean value theorem,
\begin{eqnarray*}
&P(|X-t|\leq\del)\leq P(t-2\del<X-t\leq t+2\del)=P(X\leq t+2\del)-P(X\leq
 t-2\del)\\
 &\leq 2d_{k}(X,Z)+P(Z\leq t+2\del)-P(Z\leq t-2\del)\leq2d_{k}(X,Z)+4c\del.
\end{eqnarray*}
 Therefore by the Markov inequality,
\begin{equation*}
|P\{Y\leq t\}-P\{Z\leq t\}|\leq 3d_{k}(X,Z)+4c\del+\frac{E|X-Y|^{a}}{\del^{a}}.
\end{equation*}
 The lemma follows first by taking supremum over $t\in\bbR$
and then taking $\del=||X-Y||_{a}^{\frac{a}{a+1}}.$
\end{proof}

Next, we introduce notations which appeared in  \cite{KV1} and will
be useful here, as well. Set
\begin{eqnarray}\label{3.2}
&F_{i,n,r}(x_{1},...,x_{i-1},\omega)=E(F_{i}(x_{1},...,x_{i-1},X(n))|
\cF_{n-r,n+r}),\\
&Y_{i,q_{i}(n)}=F_{i}\left(X(q_{1}(n)),...,X(q_{i}(n))\right)\nonumber\\
&\mbox{ and}\,\, Y_{i,m}=0\,\,\mbox{ if }\,\,  m\notin\{q_{i}(n)\}_{n=1}^{
\infty},\, X_{r}(n)=E(X(n)|\cF_{n-r,n+r}),\nonumber\\
 &Y_{i,q_{i}(n),r}=F_{i,q_{i}(n),r}\left(X_{r}(q_{1}(n)),...,X_{r}(q_{i-1}(n)),
\omega\right)\,\,\mbox{and}\,\,Y_{i,m,r}=0\,\,\mbox{if}\,\,m\notin\{q_{i}(n)
\}_{n=1}^{\infty}.\nonumber
\end{eqnarray}

 We will rely on the following result obtained in Lemma 4.2 of \cite{KV1}
under Assumption \ref{ass2.1} with $b=2$. Set
\[
b_{i,j}(n,l)=E(Y_{i,q_{i}(n)}Y_{j,q_{j}(n)})
\]
and
\[
\hat{s}_{i,j}(n,l)=\min(q_{i}(n)-q_{j}(l),n) \mbox{ and } s_{i,j}(n,l)=
\max(\hat{s}_{i,j}(n,l),\hat{s}_{j,i}(l,n)).
\]
Then, there exits a nonincreasing sequence $h(m)$, satisfying
\begin{equation}\label{4.1}
\sup_{n,l:  s_{i,j}(n,l)\geq m}|b_{i,j}(n,l)|\leq h(m).
\end{equation}
Moreover, for $m>L_{1}$ we can set $h(m)=C(\varpi_{q,p}(n)+\be(q,n)^{\del})$
with $n=[\frac 13m]$ and some positive constants $L_{1}$ and $C$ depending
 only on the initial parameters.

Next, we will obtain estimates of errors for approximating expectations of
the form $EG(X(n_{1}),...X(n_{s}))$, where $n_{1}<...<n_{s}$, by
 corresponding expectations with respect to corresponding product measures.
The result is similar to Lemma 4.3 from \cite{KV1} but the latter does not
provide specific estimates which we need here. First, we will recall the
inequality (3.14) from Corollary 3.6 of \cite{KV1}. Let $\cG$
and $\cH$ be sub-$\sig$ algebras of a probability
space $(\Om,\cF,P)$, $X$ be $d$-dimensional random vector
and $f=f(x,\omega)$, $x\in\bbR^{d}$ be a collection of random
variables that are measurable with respect to $\cH$
which satisfy
\begin{equation}\label{4.2}
||f(x,\omega)-f(y,\omega)||_{q}\leq C_{1}(1+|x|^{\iota}+|y|^{\iota})
|x-y|^{\ka}\mbox{ and }
||f(x,\omega)||_{q}\le C_{2}(1+|x|^{\iota}).
\end{equation}
Then, assuming that $\frac{1}{a}\geq\frac{1}{p}+\frac{\iota+2}{m}+
\frac{\del}{q}$
and $1\geq\ka>\te>\frac{d}{p}$,
\begin{eqnarray}\label{4.3}
&\| E\big(f(X,\cdot)|\cG\big)-g(X)\|_{a}\leq c\varpi_{q,p}(\mathcal{G},\cH)
(C_{1}+C_{2})^{\frac{d}{p\te}}C_{2}^{1-\frac{d}{p\te}}(1+||X||_{m}^{\iota+1})
\\
&+2c(C_{1}+C_{2})(1+2||X||_{m}^{\iota+2})||X-E\big(X|\mathcal
{G}\big)||_{q}^{\del}
\nonumber
\end{eqnarray}
where $c=c(\iota,\ka,\te,p,q,a,\del,d)>0$ depends only on
parameters in brackets and $g(x)=Ef(x,\omega)$. Assuming
that $a\geq1$, taking expectation and using the H\" older
inequality together with (\ref{4.3}) we obtain
\begin{equation}\label{4.4}
|Ef(X,\cdot)-Eg(X)|\leq R
\end{equation}
where $R$ is the right hand side of (\ref{4.3}). As a conclusion of
(\ref{4.4}) we derive the following result.

\begin{lemma}\label{lem4.1}
Suppose that Assumption 2.1 holds true. Let  $G:(\bbR^{\wp})^{n_{1}+...+n_{v}}
\to\bbR$ be a function satisfying conditions (\ref{2.5})
with $K',\ka$ and $\iota'\leq b\iota$. Suppose that the sets
$M_{i}=\{a_{i,1}<...<a_{i,n_{i}}\}\subset\bbN$ satisfy
$a_{i,n_{i}}<a_{i+1,1}$ and set $X(M_{i})=(X(a_{i,1}),...,X(a_{i,n_{i}}))$
where $i=1,...,v$ and $r=\underset{1\leq i\leq v-1}{\min}\{a_{i+1,1}
-a_{i,n_{i}}\}$.
Let $\{Y(M_{i})\}_{i=1}^{v}$ be independent copies of $\{X(M_{i})\}_{i=1}^{v}$.
Assume that $1\geq\frac{1}{p}+\frac{\iota'+2}{m}+\frac{\del}{q}$. Then
\begin{equation}\label{4.5}
|EG(X(M_{1}),...,X(M_{v}))-EG(Y(M_{1}),...,Y(M_{v}))|\leq C((\be(q,[\frac{r}
{4}]))^{\del}+\varpi_{q,p}([\frac{r}{4}]))
\end{equation}
where $C$ depends only on the initial parameters and on $\max_{i}\{n_{i}\}$,
 $v$
and $K'$.
\end{lemma}
\begin{proof}
For $i=1,...,v$ set  $z_{i}=(x_{a_{i,1}},...,x_{a_{i,n_{i}}})$,
\[
\hat{X}_{i}=(X(M_{1}),...,X(M_{i})) \mbox{ and } H^{(v)}(z_{1},...,z_{v})=
G(z_{1},...,z_{v}).
\]
Define recursively for $j=v,v-1,...,1$,
\[
H^{(j-1)}(z_{1},...,z_{j-1})=\int H^{(j)}(z_{1},...,z_{j})d\nu_{j}(z_{j})=
EH^{(j)}(z_{1},...,z_{j-1},X(M_{j})).
\]
Notice that $H^{(0)}=EG(\hat{Y}_{v})$.  For any $s>0$, set
\[
H_{s}^{(j-1)}(z_{1},...,z_{j-1})=EH^{(j)}(z_{1},...,z_{j-1},X_{[\frac{s}{4}]}
(M_{j})).
\]
Observe that since $m>b\iota$, $X(M_{j})$ has a finite $\iota'$
moment. Hence, $H^{(j)}$and $H_{r}^{(j)}$ also satisfy conditions
(\ref{2.4})-(\ref{2.5}). Thus, by the contraction of conditional
expectations and Lemma \ref{lem3.2} we obtain,
\[
|EH^{(j)}(\hat{X}_{j})-EH^{(j)}(\hat{X}_{j-1},X_{[\frac{r}{4}]}(M_{j}))|
\leq C(\be(q,[\frac{r}{4}]))^{\del}
\]
and
\[
|EH^{(j-1)}(\hat{X}_{j-1})-EH_{r}^{(j-1)}(\hat{X}_{j-1})|\leq
C(\be(q,[\frac{r}{4}]))^{\del}.
\]
Let $1\leq j\leq v$ and set $f(y,\om)=H^{(j)}(y,X_{[\frac{r}{4}]}(M_{j}))$.
Observe that condition (\ref{4.2}) is satisfied with constants
which depend only on the initial parameters since $X(M_{j})$ has
a finite $bq\iota$ moment and $q\iota'\leq bq\iota$. Taking $\cG=
\cF_{-\infty,a_{j-1,n_{j-1}}+[\frac{r}{4}]}$
and applying (\ref{4.4}) we obtain that
\[
|EH^{(j)}(\hat{X}_{j-1},X_{[\frac{r}{4}]}(M_{j}))-EH_{r}^{(j-1)}
(\hat{X}_{j-1})|\leq C'(\varpi_{q,p}([\frac{r}{2}])+\be(q,[\frac{r}{4}]))^{
\del})
\]
and therefore,
\[
|EH^{(j)}(\hat{X}_{j})-EH^{(j-1)}(\hat{X}_{j-1})|\leq C''[\varpi_{q,p}
([\frac{r}{2}])+\be(q,[\frac{r}{4}]))^{\del}].
\]
Finally, using the fact that
\[
H^{(v)}(\hat{X}_{v})-H^{(0)}=\sum_{j=1}^{v}H^{(j)}(\hat{X}_{j})-H^{(j-1)}
(\hat{X}_{j-1})
\]
 we obtain (\ref{4.5}) completing the proof.
\end{proof}

We will need the following general estimates which appeared as Lemmas 6.1
and 6.2 in earlier
preprint versions of \cite{KV1} (see arXiv:1012.2223v2) but not in its
published version so for readers' convenience we provide them here.
Consider a probability space $(\Omega, \cF, P)$ with a filtration of
$\sigma-$fields $\cG_j$. Suppose that random variables $X_j$ are $\cG_j$
measurable  and for some $2\leq p<\infty$ satisfy
\begin{equation}\label{eq1.2}
\gam_p=\sup_j \|X_j\|_p\leq
\sup_i \sum_{j\ge i} \|E[X_j|\cG_i]\|_p=A_p<\infty.
\end{equation}
We will explore the behavior of  higher  order moments for sums
$S_n=\sum_{i=1}^n X_i$ obtaining estimates of the form
$E[S_n^{2l}]\le C_{2l} n^l$
with some control on dependence of constants  $C_{2l}$ on  $\gam_{2l}$
and $A_{2l}$.
\begin{lemma}\label{Lem6.1}
Suppose $\{a_n\}$ is a sequence of nonnegative numbers such that for some
integer $l\geq 1$ and any integer $n\ge 1$,
$$
a_{n+1}\le  c\sum_{j=1}^n \sum_{r=2}^{2l} C^r a_j^\frac{2l-r}{2l}.
$$
Then
$$
a_n\le A\,n^l
$$
with $A=\max\{ 2^l c^l C^{2l}, C^{2l}, a_1\}$.
\end{lemma}
\begin{proof} We derive the above inequality by induction. It is clearly valid
for $n=1$. Assume it is valid for $j=1,2,\ldots, n.$ Then
\begin{eqnarray*}
&a_{n+1}\le c\sum_{j=1}^n \sum_{r=2}^{2l}C^r(A j^l)^\frac{2l-r}{2l}\\
&\le c\,C^2\, A^{1-\frac{1}{l}} \sum_{r=2}^{2l}  C^{r-2} A^{-\frac{r-2}{2l}}\,
\sum_{j=1}^n j^{l-1} \le A' \frac{(n+1)^l}{l}
\end{eqnarray*}
where
$$
A'=c\,C^2\, A^{1-\frac{1}{l}}  \sum_{r=0}^{2l-2}  C^r A^{-\frac{r}{2l}}
$$
and we need to pick $A$ so that $\frac{A}{l}'\le A$. In particular,
$A=\max\{ 2^l c^l C^{2l}, C^{2l}, a_1\}$ will do because
$CA^{-\frac{1}{2l}}\le 1$, $ 2\,c\, C^2 A^{-\frac{1}{l}}\le 1$ and
$$
c\,C^2\, A^{1-\frac{1}{l}}  \sum_{r=0}^{2l-2}  C^r A^{-\frac{r}{2l}}\le c
\,C^2\, A^{1-\frac{1}{l}} (2l-1)\le c\,C^2\, A^{1-\frac{1}{l}} 2l\le l \,A.
$$
\end{proof}
\begin{lemma}\label{Lem6.2} Let  the sequence $\{X_i\}$ of random variables
satisfy (\ref{eq1.2}) with $p=2l$ and some positive integer $l$. Then there is
 a constant
$c_l$ depending only on $l$ such that
$$
ES_n^{2l}\le c_l\,  A_{2l}^{2l} \, n^l.
$$
\end{lemma}
\begin{proof}
We begin by expanding  $S_{j+1}^{2l}=(S_j+X_{j+1})^{2l}$ by the binomial
theorem,
$$
S_{j+1}^{2l}= S_j^{2l}+2l S_j^{2l-1}X_{j+1}+\sum_{r=2}^{2l} {2l\choose r}
S_j^{2l-r}X_{j+1}^r
$$
and expressing
$$
S_j^{2l-1}=\sum_{i=1}^j(S_i^{2l-1}-S_{i-1}^{2l-1})\\
=\sum_{1\le i\le j} X_i\sum_{r=0}^{2l-2} S_i^{r}S_{i-1}^{2l-2-r}.
$$
This enables us to rewrite
$$
S_{j+1}^{2l}=S_j^{2l}+2l \sum_{1\le i\le j } Z_i  X_{j+1}+\sum_{r=2}^{2l}
{2l\choose r}S_j^{2l-r}X_{j+1}^r
$$
where $Z_i=X_i\sum_{r=0}^{2l-2} S_i^{r}S_{i-1}^{2l-2-r}$.
Then,

\begin{align*}
ES_{n+1}^{2l}&=EX_1^{2l}+2l \sum_{1\le i\le j \le n }E Z_i  X_{j+1}+
\sum_{j=1}^n \sum_{r=2}^{2l} {2l\choose r}ES_j^{2l-r}X_{j+1}^r\\
&=2l \sum_{1\le i \le n }E Z_i  W_i+\sum_{j=1}^n \sum_{r=2}^{2l}
{2l\choose r}ES_j^{2l-r}X_{j+1}^r
\end{align*}
where $W_i= \sum_{j=i}^n E(X_{j+1}|\cF_i)$. We note that   $\|X_i\|_{2l}\le
\gamma_{2l}\le A_{2l}$ and $\|W_i\|_{2l}\le A_{2l}$.  Hence,

\begin{align*}
E[|Z_iW_i|]&\le \|\sum_{r=0}^{2l-2} S_i^{r}S_{i-1}^{2l-2-r}\|_\frac{l}{l-1}
\|X_i\|_{2l} \|W_i\|_{2l}\\
&\le  c_l  A_{2l}^2 ((E S_i^{2l})^\frac{l-1}{l}+
(ES_{i-1}^{2l})^\frac{l-1}{l} ).
\end{align*}

Next, for $r\ge 2$,
$$
|ES_j^{2l-r}X_{j+1}^r|\le \|S_j\|_{2l}^{2l-r} \|X_{j+1}\|_{2l}^r\le
A_{2l}^r \|S_j\|_{2l}^{2l-r}.
$$
It follows that
\begin{align*}
ES_{n+1}^{2l}&\le  c_l \bigg(\sum_{j=1}^n\big( \sum_{r=2}^{2l} A_{2l}^r
\|S_j\|_{2l}^{2l-r}+A_{2l}^2\big \|S_j\|_{2l}^{2l-2}+ A_{2l}^2
\|S_{j-1}\|_{2l}^{2l-2}\big) \bigg)\\
&\le c_l  \bigg(\sum_{j=1}^n\sum_{r=2}^{2l} A_{2l}^r \|S_j\|_{2l}^{2l-r} \bigg)
\end{align*}
where $c_l$ is an  absolute constant which depends only on $l$.
The sequence $a_n=E[S_n^{2l}]$ satisfies the condition of Lemma \ref{Lem6.1}
with $c= c_l$, $C= A_{2l}$ and $a_1\le \gamma_{2l}^{2l}$ and the result
follows.
\end{proof}

\section{Limiting variance}\label{sec4}
\setcounter{equation}{0}

In this section we will prove Theorem \ref{thm2.3}.
For each $i=1,..,\ell$ set
\[
Z_{i,n}=F_{i}\left(X^{(1)}(n),...,X^{(i)}(in)\right)
\]
 and $\Sigma_{i,N}=\sum_{n=1}^{N}Z_{i,n}$ so that $Z_n=\sum_{i=1}^\ell Z_{i,n}$
 and $\Sigma_N=\sum_{i=1}^\ell\Sigma_{i,N}$. Then, under the assumption
 (\ref{2.4}) the processes $\{Z_{i,n}\}_{n\ge 0},\, i=1,...,\ell$
and $\{Z_{n}\}_{n\ge0}$ are (one sided) stationary in the wide sense.
In view of (\ref{2.14}),
\begin{equation}\label{4.1.1}
EZ_{i,n}Z_{j,m}=0\,\,\mbox{if}\,\, i\ne j\,\,\mbox{and so}\,\,
\mbox{Var}\Sig_{N}=\sum_{i=1}^{\ell}\mbox{Var}(\Sig_{i,N}).
\end{equation}
Hence, $EZ_nZ_0=\sum_{i=1}^\ell EZ_{i,n}Z_{i,0}$. In the same way as Lemma 4.2
of \cite{KV1}
provides the estimate (\ref{4.1}) with $h(m)=C(\varpi_{q,p}([\frac 13m])+
\be(q,[\frac 13m])^{\del})$ for some $C>0$ and all $m$ large enough we obtain
that for all $n$ large enough and some $C>0$ independent of $n$,
\[
|EZ_{n,i}Z_{0,i}|\leq C(\varpi_{q,p}([\frac 13n])+\be(q,[\frac 13n])^{\del}).
\]
This together with Assumption \ref{ass2.1} with $\al,\la\geq 1$ yields that
\begin{equation}\label{4.6}
\sum_{n=1}^{\infty}n|E(Z_{n}Z_{0})|<\infty.
\end{equation}
By Proposition 8.3 and Theorem 8.6 from \cite{Br} (modified for
a one sided process) if a stationary in the wide sense process satisfies
(\ref{4.6}) then $s^{2}=\lim_{n\to\infty}\frac 1n$Var$\Sig_n$ exists and
Var$\Sigma_{N}$ is unbounded if and only if $s^{2}>0$ which is equivalent
to the fact that there exists no representation of the form
$Z_{n}=V_{n+1}-V_{n}$ where $V_n,\, n\geq 0$ is a square integrable stationary
in the wide sense
process. These together with (\ref{4.1.1}) implies that $s^{2}=0$ if and
only if Var$(\Sigma_{i,N})$ is bounded for each $i=1,...,\ell$.

Next, set
\[
S_{N}=\sum_{n=1}^{N}F(X(n),...,X(\ell n)), S_{i,N}=
\sum_{n=1}^{N}F_{i}(X(n),...,X(in)),
\]
\[
N_{\ell}=N_{\ell}^{(1)}=[N(1-\frac{1}{2\ell})]+1\,\,\mbox{and}\,\,
N_{\ell}^{(i)}=[N_{\ell}^{(i-1)}(1-\frac{1}{2\ell})]+1\,\,\mbox{for}\,\,
i=1,2,3,
... \]
\[
S_{\ell,N}^{(-1)}=S_{\ell,N}=\sum_{n=1}^{N}F_{\ell}(X(n),...,X(\ell n))
 \]
and
\[
S_{\ell,N}^{(2i-1)}=\sum_{n=1}^{N_{\ell}^{(i)}-1}F_{\ell}(X(n),...,
X(\ell n)),S_{\ell,N}^{(2i)}=S_{\ell,N}^{(2i-3)}-S_{\ell,N}^{(2i-1)},
i=1,2,3....
\]
Set also $\sig_{N}^{2}=var(S_{N})$ and $s_{N}^{2}=var(\Sigma_{N}).$

Now we can write
\begin{eqnarray}\label{4.7}
&\sigma_{N}^{2}=\mbox{Var}(\sum_{i=1}^{\ell-1}S_{i,N}+S_{\ell,N}^{(1)})+
\mbox{Var}(S_{\ell,N}^{(2)})\\
&+2\mbox{Cov}(\sum_{i=1}^{\ell-1}S_{i,N}+S_{\ell,N}^{(1)},\, S_{\ell,N}^{(2)}).
\nonumber\end{eqnarray}
Observe that $N_{\ell}\geq\frac{N}{2}$. Since $\ell m-in\geq\frac{N}{2}$
whenever $i<\ell,\, n\leq N$ and $N_{\ell}\leq m\leq N$ then
$|b_{i,\ell}(n,m)|\leq h([\frac{N}{2}])$ by (\ref{4.1}). Taking into account
Assumption \ref{ass2.1} with $\al,\la\geq 1$ and the choice of the
nonincreasing function $h$ we obtain that
\begin{eqnarray}\label{4.7+}
&|\mbox{Cov}(\sum_{i=1}^{\ell-1}S_{i,N},\, S_{\ell,N}^{(2)})|\leq
\sum_{i=1}^{\ell-1}\sum_{n=1}^{N}\sum_{m=N_{\ell}}^{N}|b_{i,\ell}(n,m)|\\
&\leq\ell N^{2}h([\frac{N}{2}])\leq 16\ell\sum_{n=1}^{\infty}nh(n)<\infty.
\nonumber\end{eqnarray}
Furthermore, since $|b_{\ell,\ell}(n,m)|\leq h(m-n)$ when $n<m$ we obtain
\begin{eqnarray}\label{4.7++}
&|\mbox{Cov}(S_{\ell,N}^{(1)},\, S_{\ell,N}^{(2)})|\leq\sum_{n=1}^{N_\ell-1}
\sum_{m=N_\ell}^N h(m-n)=\sum_{n=1}^{N_\ell-1}\sum_{j=N_\ell-n}^{N-n}h(j)=\\
&\sum_{j=1}^{N_-1}\sum_{n=\max(N_\ell-j,1)}^{\min(N-j,N_\ell)}h(j)=
\sum_{j=N-N_\ell}^{N_\ell-1}(N-N_\ell)h(j)+\sum_{j=1}^{N-N_\ell}jh(j)\nonumber\\
&+\sum_{j=N_\ell-1}^{N-1}(N-j)h(j)\leq\sum_{j=1}^{\infty}jh(j)<\infty.\nonumber
\end{eqnarray}

Next, define $\Sig_{\ell,N}^{(j)}$ for $j=-1,1,2,3....$ similarly to
$S_{\ell,N}^{(j)}$ using $X^{(1)}(n),...,X^{(\ell)}(\ell n)$ in place
of $X(n),...,X(\ell n)$. Observe that $N_{\ell}^{(i-1)}(1-\frac{1}{2\ell})
+1\geq N_{\ell}^{(i)}\geq N_{\ell}^{(i-1)}(1-\frac{1}{2\ell})$
for any  $i,j\leq\ell$ and so,
\begin{equation}\label{4.1.3}
jN_{\ell}^{(i)}-(j-1)N_{\ell}^{(i-1)}\geq\frac{1}{2}N_{\ell}^{(i-1)}\,\,
\mbox{and}\,\, i\geq N_{\ell}^{(i)}- N(1-\frac{1}{2\ell})^{i}\geq 0.
\end{equation}
Applying Lemma \ref{lem4.1} for
\[
G(X(n,m),X(2n,2m),...,X(\ell n,\ell m))=F_\ell(X(n ),...,X(\ell n))
F_\ell(X(m),...,X(\ell m))
\]
where $N^{(i-1)}_\ell>n,m\geq N^{(i)}_\ell$ we obtain taking into account
(\ref{4.1.3}) that
\begin{eqnarray*}
&|EF_\ell(X(n ),...,X(\ell n))F_\ell(X(m),...,X(\ell m))-EF_\ell(X^{(1)}(n),
...,X^{(\ell)}(\ell n))\\
&\times F_\ell(X^{(1)}(m),...,X^{(\ell)}(\ell m))|\leq
C\gam([\frac {N_\ell^{(i-1)}}{8}])
\end{eqnarray*}
where $\gam(n)=\varpi_{q,p}(n)+\be^{\del}(q,n)$ and $C>0$ depends only on
the initial parameters. Hence for all $i\geq 1$,
\begin{eqnarray}\label{4.8}
&|\mbox{Var}(S_{\ell,N}^{(2i)})-\mbox{Var}(\Sigma_{\ell,N}^{(2i)})|
\leq C(N_{\ell}^{(i-1)}-N_{\ell}^{(i)})^{2}\gam([\frac{(N_{\ell}^{(i-1)}}{8}])\\
&\leq 64C\sup_{m\geq 1}m^{2}\gam(m)\leq 256C\sum_{n=1}^\infty n\gam(n)
= c_{1}\nonumber
\end{eqnarray}
where $N_{\ell}^{(0)}=N$ and $c_{1}>0$ depends only on the initial parameters
and the expressions (\ref{2.8}) and (\ref{2.9}).

Next, assume that $s_N^2$ is bounded. Then by (\ref{4.1.1}) we see that
Var$(\Sigma_{i,N})$ is bounded in $N$ for each $i=1,...,\ell$. Proving one
direction of Theorem  \ref{thm2.3} we will derive from here by induction
in $j$ that for each $j$ there exists $C_j>0$ such that for all $N\geq 2$,
\begin{equation}\label{4.1.4}
\mbox{Var}(\sum_{i=1}^{j}S_{i,N})\leq C_{j}\ln^{2}N.
\end{equation}
When $j=1$ we have Var$S_{1,N}$=Var$\Sigma_{1,N}$ which is bounded
if $s_N^2$ is bounded. Now suppose that (\ref{4.1.4}) holds true for all $j$
up to $\ell-1$ and prove it for $j=\ell$. Recall that $Z_{\ell,n},\, n\geq 0$
is a stationary in the wide sense process, and so
\[
\mbox{Var}(\Sigma_{\ell,N}^{(2i)})=\mbox{Var}(\sum^{N_\ell^{(i-1)}-1}
_{n=N_\ell^{(i)}}Z_{\ell,n})=\mbox{Var}\Sigma_{\ell,N_\ell^{(i-1)}
-N_\ell^{(i)}}
\]
and the latter expression is bounded in view of our assumption on $s_N^2$.
This together with (\ref{4.8}) yields
\begin{equation}\label{4.1.5}
\mbox{Var}(S^{(2i)}_{\ell,N})\leq\mbox{Var}(\Sigma_{\ell,N}^{(2i)})+c_1\leq
c_2
\end{equation}
for some $c_2>0$ independent of $N$. Now by (\ref{4.7})--(\ref{4.7++}),
(\ref{4.1.5}) and the induction hypothesis
\begin{equation}\label{4.1.6}
\sig_N^2=\mbox{Var}(\sum_{i=1}^\ell S_{i,N})\leq c_3+2C_{\ell -1}\ln^2N+
2\mbox{Var}(S^{(1)}_{\ell,N})
\end{equation}
for some $c_3>0$ independent of $N$.

Next, applying the above definitions recursively for any $i$ such that
$N^{(i)}_\ell\geq 2$ we can write
\begin{equation}\label{4.1.7}
S^{(1)}_{\ell,N}= S^{(3)}_{\ell,N}+S^{(4)}_{\ell,N}=S^{(2i-1)}_{\ell,N}+
\sum_{j=2}^iS^{(2j)}_{\ell,N}.
\end{equation}
Hence,
\[
\mbox{Var}S^{(1)}_{\ell,N}\leq 2\mbox{Var}S^{(2i-1)}_{\ell,N}+2i
\sum_{j=2}^i\mbox{Var}S^{(2j)}_{\ell,N}
\]
where we use that $(\sum_{j=1}^ma_j)^2\leq m\sum_{j=1}^ma_j^2$. By (\ref{4.1.3})
we can choose $i=M\ln N$ for some fixed $M>0$ so that $2\leq N^{(i)}_\ell\leq
i+4$. Then Var$S^{(2i-1)}_{\ell,N}\leq C'(1+\ln^2(N))$  for some $C'>0$
independent of
$N$ and we obtain from
(\ref{4.1.5}) that Var$S^{(1)}_{\ell,N}\leq\tilde C(1+\ln^2N)$ for some
$\tilde C>0$ independent of $N$ which together with (\ref{4.1.6}) yields
(\ref{4.1.4}) with $j=\ell$.

Next, we will prove Theorem \ref{thm2.3} in the other direction assuming
that $s^2_N$ is unbounded which, as explained above, is equivalent to the
linear in $N$ growth of $s^2_N$ and to the fact that the corresponding
limiting variance $s^2$ is positive. Our goal is to show that then
\begin{equation}\label{4.1.8}
\sig^2=\lim_{N\to\infty}\frac 1N\sig^2_N>0.
\end{equation}
Recall, that the existence of the limit in (\ref{4.1.8}) follows from
\cite{KV1} and only its positivity should be proved in our situation.
The proof will proceed again by induction in $\ell$. For $\ell=1$ we
have $S_N=\Sigma_N$, and so if Var$\Sigma_N$ grows linearly in $N$ then
the same is true for Var$S_N$. Now suppose that we already established
for each $j=1,2,...,\ell-1$ that if Var$(\sum_{i=1}^j\Sigma_{i,N})$ grows
linearly in $N$ then the same is true for Var$(\sum_{i=1}^jS_{i,N})$ and
now we will prove this for $j=\ell$.
Indeed, assume that $s_N^2=$Var$(\sum_{i=1}^\ell\Sig_{i,N})$ grows linearly
in $N$. Then by (\ref{4.1.1}) either Var$(\sum_{i=1}^{\ell-1}\Sig_{i,N})$ or
Var$(\Sig_{\ell,N})$ grow linearly in $N$. In the latter case we obtain also
that Var$(\Sigma^{(2)}_{\ell,N})$ grows linearly in $N$ in view of
stationarity in the wide sense of $Z_{\ell,n},\, n\geq 0$. Then by (\ref{4.8})
we see that Var$(S^{(2)}_{\ell,N})$ grows linearly in $N$. This together
with (\ref{4.7})--(\ref{4.7++}) yields that $\sig_N^2$ grows at least
linearly in $N$ but since by \cite{KV1} a (finite) limit $\lim_{N\to\infty}
\frac 1N\sig_N^2$ exists we conclude that in this case $\sig^2_N$ grows
linearly in $N$ as required.

Now suppose that Var$(\Sigma_{\ell,N})$ is bounded while Var$(
\sum_{i=1}^{\ell-1}\Sigma_{i,N})$ grows linearly in $N$. Then
Var$(\Sigma^{(2i)}_{\ell,N})$ for all $i$ are also bounded by stationarity
of $Z_{\ell,n},\, n\geq 0$ in the wide sense which together with (\ref{4.8})
yields that Var$(S^{(2i)}_{\ell,N})$ are also bounded for all $i$.
Using again the
representation (\ref{4.1.7}) we conclude as before that Var$S^{(1)}_{\ell,N}
\leq\tilde{\tilde C}(1+\ln^2N)$ for some $\tilde{\tilde C}>0$ independent of $N$.
Since Var$(\sum_{i=1}^{\ell-1}\Sigma_{i,N})$ grows linearly in $N$ then by
the induction hypothesis Var$(\sum_{i=1}^{\ell-1}S_{i,N})$ grows linearly in
$N$, as well. It follows that
\begin{eqnarray}\label{4.1.9}
&\sig^2_N=\mbox{Var}(\sum_{i=1}^\ell S_{i,N})=\mbox{Var}(\sum_{i=1}^{\ell-1}
 S_{i,N})+\mbox{Var}S_{\ell,N}\\
 &+2\mbox{Cov}(\sum_{i=1}^{\ell-1}S_{i,N},S_{\ell,N})\geq\mbox{Var}
 (\sum_{i=1}^{\ell-1}S_{i,N})\nonumber\\
 &-2(\mbox{Var}(\sum_{i=1}^{\ell-1}S_{i,N}))^{1/2}\big((\mbox{Var}
 S^{(1)}_{\ell,N})^{1/2}+(\mbox{Var}S^{(2)}_{\ell,N})^{1/2}\big)\nonumber\\
 &\geq\mbox{Var}(\sum_{i=1}^{\ell-1}S_{i,N})(1-\hat CN^{-1/2}(\ln N+1))
 \nonumber\end{eqnarray}
 for some $\hat C>0$ independent of $N$. Hence, $\sig_N^2$ grows at least
 linearly in $N$ but, again, since finite limit $\lim_{N\to\infty}\frac 1N
 \sig_N^2$ exists $\sig_N^2$ grows, in fact, linearly in $N$ completing the
 proof of Theorem \ref{thm2.3}.  \qed

\section{Convergence estimates}\label{sec5}
\setcounter{equation}{0}

In this section we introduce martingale approximation technique which
is similar but a bit different from \cite{KV1}. Then we study the quadratic
variation of the constructed martingale and use it to prove Theorem
\ref{thm2.4}. The following representations from (5.2) in \cite{KV1} will be
useful here, as well.
\begin{eqnarray}\label{5.1}
&\hskip1.cm Y_{i,n}=Y_{i,n,1}+\sum_{r=1}^{\infty}[Y_{i,n,2^{r}}-
Y_{i,n,2^{r-1}}],\,\zeta_{i,N,0}(t)=\frac{1}{\sqrt{N}}\sum_{1\leq n\leq Nt}
Y_{i,n,1},\\
&\zeta_{i,N,r}(t)=\frac{1}{\sqrt{N}}\sum_{1\leq n\leq Nt}[Y_{i,n,2^{r}}-Y_{i,n,
2^{r-1}}],\,  r\geq1,\,\,\xi^{(u)}_{i,N}(t)=\sum_{r=0}^{u}\zeta_{i,N,r}(t)
\nonumber\\
&\mbox{and}\,\, \xi_{i,N}(t)=\sum_{r=0}^{\infty}\zeta_{i,N,r}(t)=
\frac{1}{\sqrt{N}}\sum_{1\leq n\leq Nt}Y_{i,n}.\nonumber
\end{eqnarray}

The convergence of series in (\ref{5.1}) follows from \cite{KV1}.
Similarly to Proposition 5.8 of \cite{KV1} we derive from Corollary 3.6(ii)
there that for any $r\in\bbN$, $1\leq i\leq\ell$ and $l\leq n+r$,
\begin{equation}\label{5.1.1}
||E(Y_{i,n,r}|\cF_{-\infty,l})||_{b}\leq c_{r}(n-l)
\end{equation}
for some sequence $c_r$ satisfying
\begin{equation}\label{5.1.2}
C_{r}=\sum_{m=0}^{\infty}c_{r}(m)\leq Cr
\end{equation}
where $b$ comes from Assumption \ref{ass2.1} and, recall, that $b\geq 4$ in
Theorem \ref{thm2.4} while $C>0$ depends on the initial parameters and
on the expressions (\ref{2.8}) and (\ref{2.9}). Furthermore, similarly to
the proof of Proposition 5.9 from \cite{KV1} we obtain that
\begin{equation}\label{5.1.3}
\sum_{n\geq r}\sup_{N\geq1}||\sup_{0\leq t\leq T}|\zeta_{i,N,n}(t)|\:||_{2}
\leq CT\sum_{n=2^{r-1}}^{\infty}(\be(q,n))^{\del}
\end{equation}
where $C>0$ depends only on the initial parameters and
the expressions (\ref{2.8}) and (\ref{2.9}). Observe that since
\[
\sum_{n=m}^{\infty}(\be(q,n))^{\del}\leq m^{-\la}\sum_{n=m}^{\infty}
n^{\la}(\be(q,n))^{\del}
\]
then under Assumption 2.1 with $b\geq 2$ we obtain from (\ref{5.1.3}) that
\begin{equation}\label{5.2}
||\underset{0\leq t\leq T}{\sup}|\xi_{i,N}(t)-\xi_{i,N}^{(u)}(t)|
\thinspace||_{2}\leq CT2^{-u\la}
\end{equation}
where $C>0$ depends only on the initial parameters and
the expressions (\ref{2.8}) and (\ref{2.9}).

\subsection{Martingale approximations}\label{subsec5.1}

For any fixed $n,u\in\bbN$ and $1\leq i\leq \ell$ set
\begin{equation*}
W_{i,n,2^{u}}=Y_{i,n,2^{u}}+R_{i,n,u}-R_{i,n-1,u}
\end{equation*}
where $R_{i,v,u}=\sum_{s\geq v+1}E(Y_{i,s,2^{u}}|\cF_{-\infty,v+2^{u}})$.
Clearly, $\{W_{i,n,2^{u}}\}_{n\geq1}$ is a martingale difference sequence
with respect to the filtration $\{\cF_{-\infty,n+2^{u}}\}_{n\geq1}$.
For any $1\leq i\leq\ell$ and $u(N)=[\frac{\log_{2}(N)}{\la+8}]$
define the truncated martingale differences $W_{i,n}^{(N)}=
\bbI_{\{n\leq iN\}}W_{i,n,2^{u(N)}}$ and $W_n^{(N)}=\sum_{i=1}^{\ell}
W_{i,n}^{(N)}$ where $\bbI_A=1$ if an event $A$ occurs and $=0$, otherwise.
Set $M_{i,n}^{(N)}=\sum_{m=1}^{n}W_{i,m}^{(N)}$
and
\[
M_n^{(N)}=\sum_{m=1}^{n}W_m^{(N)}=\sum_{i=1}^{\ell}M_{i,n}^{(N)}.
\]
When $N$ is fixed the sequence $M_{n}^{(N)},\, n\geq 1$ is a
martingale with respect to the filtration $\{\cF_{-\infty,n+2^{u(N)}}
\}_{n\geq1}$ and when $N$ changes we have a martingale array. Taking
into account that $\xi_{N}(t)=\sum_{i=1}^{\ell}\xi_{i,N}(t)$ we obtain
 by (\ref{5.1.2}), (\ref{5.1.3}) and (\ref{5.2}) that
\begin{equation}\label{5.3}
||\xi_{N}(1)-\frac1{\sqrt{N}}M_{N\ell}^{(N)}||_{2}\leq
C(N^{-\frac{\la}{\la+8}}+N^{-\frac{\lambda+6}{2(\lambda+8)}})
\end{equation}
where $C>0$ depends only on the initial parameters and
 the expressions (\ref{2.8}) and (\ref{2.9})

\subsection{Quadratic variation estimates}\label{subsec5.2}

Our goal in this subsection is to obtain the following result.

\begin{proposition}\label{prop5.1}
Suppose that Assumption \ref{ass2.1} holds true with $\al,\la>1$. Let
$1\leq i\leq j\leq\ell$ and set
\[
Z_{n}=Z_n^{(i,j,N)}=W_{i,n,2^{u(N)}}W_{j,n,2^{u(N)}}.
\]
Then
\begin{equation}\label{5.3.1}
||\frac{1}{N}\sum_{n=1}^{iN}Z_{n}-iD_{i,j}||_{2}\leq CN^{-\frac{1}{4}
\min(\min(\al,\la)-1,\frac{\la}{\la+8})}.
\end{equation}
where $D_{i,j}$ was introduced in Theorem \ref{thm2.2} and $C$ depends only
on the initial parameters and the expressions (\ref{2.8}) and (\ref{2.9}).
\end{proposition}

We prove this proposition in several steps formulated as separate lemmas.
\begin{lemma}\label{lem5.2}
Let $1\le i\leq\ell$. Suppose that $\{n_{k}\}_{k=1}^{\infty}$ is a
strictly increasing sequence of natural numbers. Then, for any $u,m,k\in\bbN$
such that $k\leq\frac{b}{2}$,
\[
||\sum_{s=v}^{v+m-1}Y_{i,in_{s},2^{u}}||_{2k},||\sum_{s=v}^{v+m-1}
Y_{i,in_{s}}||_{2k}\leq C\sqrt{m}
\]
where $C$ depends only on the initial parameters and the expressions
(\ref{2.8})-(\ref{2.9}).
\end{lemma}
\begin{proof}
Let $s>s'$. Taking $l,n\in\bbN$ which satisfy $l\geq(i-1)n$
and $in\geq l+2s$ we can derive from Theorem 3.4 in \cite{KV1} that
\begin{equation*}
||E(Y_{i,in,s}|\cF_{-\infty,l+s})-E(Y_{i,in,s'}|\cF_{-\infty,l+s})||_{b}
\leq C_{1}\varpi_{q,p}(in-l-2s)(\be(q,s'))^{\del}
\end{equation*}
for some $C_1>0$ depending only on the initial parameters.
(see the proof of Lemma 3.11 in the early preprint version arXiv:1012.2223v2
of \cite{KV1}).
On the other hand, if $l<(i-1)n$ and $2s<n$ then by the contraction
of the conditional expectations similarly to the above,
\begin{eqnarray*}
&||E(Y_{i,in,s}-Y_{i,in,s'}|\cF_{-\infty,l+s})||_{b}\leq||E(Y_{i,in,s}-
Y_{i,in,s'}|\cF_{-\infty,(i-1)n+s})||_{b}\\
&\leq C_{2}\varpi_{q,p}(n-2s)(\be(q,s'))^{\del}
\end{eqnarray*}
for some $C_2>0$ depending only on the initial parameters.
Let $k,r,m\in\bbN$ and set $s=2^{r+1}$, $s'=2^{r}$, $l=in_{m}$ and $n=n_{k}$.
Since $n_k-n_{k'}\geq k-k'$ if $k\geq k'$ there exist no more than $4s=2^{r+3}$
natural numbers $k\geq m$ which do not satisfy either $in_m\geq (i-1)n_k$ and
$in_k\geq in_m+2^{r+2}$ or $in_m<(i-1)n_k$ and $n_k>2^{r+2}$, i.e. for each $m$
we can use one of two inequalities above with such $s,s',l$ and $n=n_k\geq n_m$
except for at most $2^{r+3}$ of $k$'s.
Using again Theorem 3.4 from \cite{KV1} (or Lemma 3.12 from the above mentioned
 preprint), the contraction of the conditional
 expectations  to bound those (at most) $2^{r+3}$ summands,
the estimates above and the fact that $\sum_{n\in\bbN}\varpi_{q,p}(n)<\infty$
we obtain
\[
\sup_{m}\sum_{k\geq m}||E(Y_{i,in_{k},2^{r+1}}-Y_{i,in_{k},2^{r}}|\cF_{in_{m}+
2^{r+1}})||_{b}\leq C_{3}2^{r}\left(\be(q,2^{r})\right)^{\del}
\]
for some $C_3>0$ depending only on the initial parameters and the expression
(\ref{2.8}).
Applying Lemma \ref{Lem6.2} with
\[
S_{z}=Y_{i,in_{s+z-1},2^{r}}-Y_{i,in_{s+z-1},2^{r+1}}
\]
 yields for $k\leq\frac{b}{2}$ that
\begin{equation}\label{5.7}
||\sum_{s=v}^{v+m-1}Y_{i,in_{s},2^{r}}-Y_{i,in_{s},2^{r+1}}||_{2k}\leq
C_{4}\sqrt{m}2^{r}\left(\be(q,2^{r})\right)^{\del}
\end{equation}
for some $C_4>0$ depending only on the initial parameters and the expression
(\ref{2.8}).
Since $Y_{i,in_{s},2^{u}}=Y_{i,in_{s}}+
\sum_{r=u}^{\infty}Y_{i,in_{s},2^{r}}-Y_{i,in_{s},2^{r+1}}$ then
we obtain for any $u>0$ that
\begin{equation}\label{5.8}
||\sum_{s=v}^{v+m-1}Y_{i,in_{s},2^{u}}||_{2k}\leq C_{4}\sqrt{m}
\sum_{r=u}^{\infty}2^{r}\left(\be(q,2^{r})\right)^{\del}+||\sum_{s=v}^{v+m-1}
Y_{i,in_{s}}||_{2k}.
\end{equation}
Since $Y_{i,in}=Y_{i,in,1}+\sum_{r=1}^{\infty}(Y_{i,in,2^{r}}-Y_{i,in,
2^{r-1}})$ almost surely then by (\ref{5.7}) and (\ref{5.1.1})
applied with $r=1$ together with Lemma \ref{Lem6.2}  applied
with $\{Y_{i,in_{s},1}\}_{s=v}^{\infty}$ we obtain  $||\sum_{s=v}^{v+m-1}Y_{i,
in_{s}}||_{2k}\leq C\sqrt{m}$ and the present lemma follows by (\ref{5.8}).
\end{proof}

\begin{lemma}\label{lem5.3}
Suppose that Assumption 2.1 holds true with $\al,\la>1$. Let $1\leq i,j\leq
\ell$.
 Let $N\in\bbN$. Set $l=l(N)=[N^{\frac{1}{2}}]$, $N_{l}=[\frac{N}{l}]$,
\[
V_{N,r}=V_{r}=(\sum_{n\in B_{r}}Y_{i,n,2^{u(N)}})(\sum_{n\in B_{r}}
Y_{j,n,2^{u(N)}})
\]
 and $U_{r}=V_{r}-EV_{N,r}$, where $B_{r}=\mathbb{\bbN}\bigcap(l(r-1),lr].$
Then,
\begin{eqnarray*}
||\frac{1}{N}\sum_{r=1}^{N_{l}}U_{N,r}||_{2}\leq CN^{-\frac{1}{4}
\min(\min(\al,\la)-1,\frac{\la}{\la+8})}.
\end{eqnarray*}
where $C$ depends on the initial parameters and the expressions (\ref{2.8}) and
(\ref{2.9}).
\end{lemma}
\begin{proof}
For any $1\leq r\leq N_{l}$ set
\[
A_{r}=\{1\leq r'\leq N_{l}\,:\thinspace\min\{|sn-tn'|:\thinspace1\leq s,t
\leq \ell^{2},n\in B_{r},n'\in B_{r'}\}<l\}.
\]
If for some $1\leq t,s\leq \ell^{2}$, $sr'l\leq t(r-2)l$ or $s(r'-1)l
\geq t(r+1)l$
then for any $n\in B_{r}$ and $n'\in B_{r'}$ $|tn-sn'|\geq l$.
Hence, for any $r'\in A_{r}$, there exist $1\leq t,s\leq \ell^{2}$
satisfying $sr'l>t(r-2)l$ and $s(r'-1)l<t(r+1)l$. Therefore, $|r'-\frac{tr}{s}|
<\max(\frac{2t}{s},\frac{t}{s}+1)\leq2\ell$
and hence $|A_{r}|\leq 4\ell^{5}$. Next,
\begin{eqnarray*}
(||\sum_{r=1}^{N_{l}}U_{r}||_{2})^{2}=\sum_{r_{1}=1}^{N_{l}}E(
\sum_{r_{2}\in A_{r_{1}}}\prod_{k=1}^{2}U_{r_{k}})+
\sum_{r_{1}=1}^{N_{l}}E(\sum_{r_{2}\notin A_{r_{1}}}\prod_{k=1}^{2}
U_{r_{k}})=J_{1}+J_{2}.
\end{eqnarray*}
First, for any $1\leq r_{1},r_{2}\leq N_{l}$, by Lemma \ref{lem5.2}
and the Cauchy-Schwarz inequality, $|E(U_{r_{1}}U_{r_{2}})|
\leq||U_{r_{1}}||_{2}||U_{r_{2}}||_{2}\leq CN$
and hence,
\begin{eqnarray}\label{5.9}
|\frac{1}{N^{2}}J_{1}|\leq\frac{C}{N}\sum_{r=1}^{N_{l}}|A_{r}|\leq
\frac{4\ell^5C}{\sqrt{N}}=\frac{C_{0}}{\sqrt{N}}.
\end{eqnarray}

Next, let $1\leq r_{1},r_{2}\leq N_{l}$ and suppose that $r_{2}\notin
A_{r_{1}}.$
Since $r_{1}\notin A_{r_{2}}$ assume without loss of generality
that $r_{1}<r_{2}$. Now we estimate $|E(U_{r_{1}}U_{r_{2}})|$. First, by
 (\ref{5.1}) $Y_{i,n,2^{u}}=Y_{i,n}+\sum_{r=u}^{\infty}Y_{i,n,2^{r}}-
Y_{i,n,2^{r+1}}$. Thus, for
$s=1,2$ the Cauchy-Schwarz inequality, Lemma \ref{lem5.2} and (\ref{5.7})
imply that
\begin{eqnarray*}
&||V_{r_{s}}-\hat{V}_{r_{s}}||_{2}\leq||\sum_{n\in B_{r_{s}}}Y_{i,n}||_{4}
||\sum_{n\in B_{r_{s}}}(Y_{j,n}-Y_{j,n,2^{u}})||_{4}\\
&\leq||\sum_{n\in B_{r_{s}}}Y_{j,n,2^{u}}||_{4}||\sum_{n\in B_{r_{s}}}
(Y_{i,n}-Y_{i,n,2^{u}})||_{4}\leq C_{1}\sqrt N2^{-u(N)\la},
\end{eqnarray*}
where
\[
\hat{V}_{r}=(\sum_{n\in B_{r}}Y_{i,n})(\sum_{n\in B_{r}}Y_{j,n}).
\]
Hence, by the above, Lemma \ref{lem5.2} and the Cauchy-Schwarz inequality,
\begin{eqnarray}\label{5.10}
&||V_{r_{1}}V_{r_{2}}-\hat{V}_{r_{1}}\hat{V}_{r_{2}}||_{1}
\leq||V_{r_{1}}||_{2}||V_{r_{2}}-\hat{V}_{r_{2}}||_{2}+
||V_{r_{2}}||_{2}||\hat{V}_{r_{1}}-V_{r_{1}}||_{2} \\
&\leq C_{2}N2^{-u(N)\la}\leq C_{3}N^{1-\frac{\la}{\la+8}}.\nonumber
\end{eqnarray}
In view of (\ref{5.10}), it suffices to estimate $cov(\hat{V}_{r_{1}},
\hat{V}_{r_{2}})$.
We show that Lemma \ref{lem4.1} is applicable. For any $n_{s},m_{s}\in
B_{r_{s}},\thinspace s=1,2$
observe at $Y_{i,n_{1}}Y_{j,m_{1}}Y_{i,n_{2}}Y_{j,m_{2}}.$ Then it
vanishes unless  $i$ divides $n_{s}$ and $j$ divides $m_{s}$
for $s=1,2$ and we can write $n_{1}=in'_{1},n_{2}=in'_{2},m_{1}=jm'_{1},
m_{2}=jm'_{2}$.

Set
\[
\gam_1=(\{sn'_{1}\}_{s=1}^{i}\cup\{sm'_{1}\}_{s=1}^{j})\,\,\mbox{and}\,\,
\gam_2=(\{sn'_{2}\}_{s=1}^{i}\cup\{sm'_{2}\}_{s=1}^{j}).
\]
 By ordering the set $\gam_1\cup\gam_2$ and considering the jump points from
$\gam_{1}$
 to $\gam_{2}$ (or vice versa) we can represent this set as a
disjoint
union of blocks with distances between them at least $\frac{l}{\ell^{2}}$ and
which are contained in $\gam_{1}$ or in $\gam_{2}$. Applying
Lemma \ref{lem4.1} first with $Y_{i,n_{1}}Y_{j,m_{1}}Y_{i,n_{2}}Y_{j,m_{2}}$
and then with $Y_{i,n_{1}}Y_{j,m_{1}}$ and $Y_{i,n_{2}}Y_{j,m_{2}}$
separately yields
\[
|E(Y_{i,n_{1}}Y_{j,m_{1}}Y_{i,n_{2}}Y_{j,m_{2}})-E(Y_{i,in_{1}}Y_{j,m_{1}})
E(Y_{i,n_{2}}Y_{j,m_{2}})|\leq C_{4}\gam(\frac{l}{4\ell^2})
\]
where $\gam(n)=\vp_{q,p}(n)+(\be(q,n))^\del$.
Finally, by ($\ref{5.10}$), the fact that $l^{2}\gam(\frac{l}{4\ell^2})\leq
cl^{-(\min(\al,\la)-1)}$
and the above inequality we see that
\begin{equation}\label{5.11}
\frac{1}{N^{2}}|J_{2}|\leq C_5\left(N^{-\frac{\la}{\la+8}}+N^{-\frac{(\min(\al,
\la)-1}{2})}\right)
\end{equation}
and the lemma follows by (\ref{5.9})-(\ref{5.11}).
\end{proof}

\begin{lemma}\label{lem5.4}
Suppose that Assumption 2.1 holds with $\al,\la>1$. Let $N\in\bbN$ and $1\leq i,
j\leq\ell$. Then
\begin{eqnarray*}
||\frac{1}{N}\sum_{n=1}^{N}(Z_{n}-E(Z_{n}))||_{2}\leq CN^{-\frac{1}{4}
\min(\min(\al,\la)-1,\frac{\la}{\la+8})}
\end{eqnarray*}
where $C$ depends only on the initial parameters and $Z_n=Z_n^{(i,j,N)}$ is
defined in
 Proposition \ref{prop5.1}.
\end{lemma}
\begin{proof}
Fix $N\in\bbN$ and let $l=[\sqrt{N}]$, $u=u(N)$. For any
$r>0$ set $\mathcal{G}_{r}=\cF_{-\infty,rl+2^{u}}$,
\begin{eqnarray*}
G_{r}=(\sum_{n\in B_{r}}W_{i,n,2^{u}})(\sum_{n\in B_{r}}W_{j,n,2^{u}})
\mbox{ and \ensuremath{T_{r}=G_{r}-\sum_{n\in B_{r}}Z_{n}}}
\end{eqnarray*}
where $B_r$ is defined in Lemma \ref{lem5.3}.
Clearly $\{T_{r}\}_{r=1}^{\infty}$ is a martingale differences sequence with
respect to the filtration $\{\mathcal{G}_{r}\}_{r=1}^{\infty}$. Then by
(\ref{5.1.1}), the
triangle inequality, the Cauchy-Schwarz inequality, Lemma \ref{lem5.2}
 and using the fact that $b\geq4$ and $2^{u}\leq\sqrt{l}$,
\begin{eqnarray}\label{5.12}
&\hskip1.cm||G_{r}-V_{r}||_{2}\leq||(\sum_{n\in B_{r}}Y_{i,n,2^{u}})
(\sum_{n\in B_{r}}Y_{j,n,2^{u}}-W_{j,n,2^{u}})||_{2}+\\
&\hskip1.cm||(\sum_{n\in B_{r}}W_{j,n,2^{u}})(\sum_{n\in B_{r}}
Y_{i,n,2^{u}}-W_{i,n,2^{u}})||_{2}\leq C_{1}\sqrt{l}2^{u}\leq
C_{2}N^{\frac{1}{4}+\frac{1}{\la+8}}\nonumber
\end{eqnarray}
and hence
\[
||\sum_{r=1}^{N_{l}}(G_{r}-V_{r})||_{2}\leq C_{2}N^{\frac{3}{4}+\frac{1}
{\la+8}}.
\]
By Lemma \ref{lem5.2} and (\ref{5.12}),
\[
||G_{r}||_{2}\leq||V_{r}||_{2}+C_{3}\sqrt{l}2^{u}\leq C_{4}N^{\frac{1}{2}}.
\]
By (\ref{5.1.2}), $||Z_{n}||_{2}\le C_{5}2^{2u}.$ Thus, $||T_{r}||_{2}\leq
2^{2u}
C_{6}N^{\frac12}$.
Therefore, by the martingale orthogonality property and since $N-lN_{l}\leq l,$
\begin{eqnarray}\label{5.13}
&||\sum_{n=1}^{N}Z_{n}-\sum_{r=1}^{N_{l}}G_{r}||_{2} & \leq C_{5}l2^{2u}+
||\sum_{r=1}^{N_{l}}T_{r}||_{2}\\
&=C_{5}l2^{2u}+(\sum_{r=1}^{N_{l}}||T_{r}||_{2}^{2})^{\frac{1}{2}} &
\leq C_{7}N^{\frac{3}{4}+\frac{2}{\la+8}}.
\nonumber
\end{eqnarray}
The lemma follows by (\ref{5.12}) and (\ref{5.13}), writing
$\sum_{n=1}^{N}Z_{n}$ as sum of those two differences and then applying
Lemma \ref{lem5.3}.
\end{proof}
\begin{proof}[\textbf{Proof of Proposition \ref{prop5.1}}]
First write
\begin{eqnarray*}
\sum_{n=1}^{iN}E(Z_{n})=E((\sum_{n=1}^{iN}W_{i,n,2^{u}})(\sum_{n=1}^{iN}
W_{j,n,2^{u}})).
\end{eqnarray*}
By the same reason as in (\ref{5.12})
and since $\frac{2^{u(N)}}{\sqrt{N}}\leq2N^{-\frac{\la+6}{2(\la+8)}}\leq2
N^{-\frac{
\la}
{2(\la+8)}}$,
\begin{equation*}
|\frac{1}{N}\sum_{n=1}^{iN}E(Z_{n})-E(\xi_{i,N}^{(u)}(i)\xi_{j,N}^{(u)}(i))|
  \leq C_{1}N^{-\frac{\la}{2(\la+8)}},\,\, u=u(N).
\end{equation*}
By (\ref{5.2}), Lemma \ref{lem5.2} and the Cauchy-Schwarz inequality,
\begin{equation*}
||\xi_{i,N}^{(u)}(i)\xi_{j,N}^{(u)}(i)-\xi_{i,N}(i)\xi_{j,N}(i)||_{2}\leq
 C_{2}2^{-u\la}\leq C_{3}N^{-\frac{\la}{2(\la+8)}}.
\end{equation*}
As in \cite{Ki2}, $|E(\xi_{i,N}(i)\xi_{j,N}(i))-iD_{i,j}|\leq C_{4}$
 for some constant which depends on the initial parameters and on
the expressions (\ref{2.8}) and (\ref{2.9}). Thus,
\begin{equation}\label{5.14}
||\frac{1}{N}\sum_{n=1}^{iN}Z_{n}-iD_{i,j}||_{2}\leq\frac{1}{N}||\sum_{n=1}^{iN}
Z_{n}-E(Z_{n})||_{2}+C_{5}N^{-\frac{\la}{2(\la+8)}}.
\end{equation}
Now we estimate the first term on the right hand side of (\ref{5.14})
by Lemma \ref{lem5.4} applied with $iN$ and the lemma follows.
\end{proof}

\subsection{Proof of Theorem \ref{thm2.4}}\label{subsec5.3}

Let $M_n=M_{n}^{(N)},\, n=1,...,\ell N$ and $W_n=W_n^{(N)},\, n=1,...,\ell N$
be the martingale array and the corresponding martingale differences from
Section \ref{subsec5.1}. Let $\sig^2>0$ be the limiting variance. Applying
Theorem 2 from \cite{HH}  with $\del=1$, and $\ve=\sig^{-\frac{8}{5}}||
N^{-1}\sum_{n=1}^{N\ell} W_{n}^{2}-\sig^{2}||_{2}^{\frac{4}{5}}$, taking into
account the Markov inequality and that always $EZ^{2}\bbI_{\{|Z|>1\}}\leq
EZ^{4}$  yields
\begin{eqnarray}\label{5.15}
&d_{K}(\cL(N^{-1/2}M_{\ell N}),\,\cN(0,\sig^2))=  d_{K}(\cL(N^{-1/2}\sig^{-1}
M_{\ell N}),
\,\cN(0,1))\\
&\leq A(N^{-2/3}\sig^{-4/3}V_{4,N}^{\frac{1}{3}} +N^{-2/5}\sig^{-\frac{4}{5}}
V_{4,N}^{\frac{1}{5}}+\sig^{-\frac{4}{5}}\|\frac 1N\sum_{n=1}^{N\ell}
W_{n}^{2}-\sig^{2}||_{2}^{\frac{2}{5}})\nonumber
\end{eqnarray}
where $A>0$ is an absolute constant and $V_{4,N}=\sum_{n=1}^{N\ell} EW_{n}^{4}$.

Next, by (\ref{5.1.1}), (\ref{5.1.2}) and the formulas for $W_{i,n,2^u},\,
W_{i,n}^{(N)}$ and $W_n^{(N)}$ we obtain that $||W^{(N)}_{n}||_{b}\leq
\bar{C}2^{u(N)}$ for some constant $\bar C>0$ independent of $N$,
 and so $V_{4,N}\leq N\bar C^42^{4u(N)}$.
Finally, (\ref{5.3}) and (\ref{5.15}) together with  Proposition \ref{prop5.1}
 and Lemma \ref{lem3.3} yields the first assertion of Theorem \ref{thm2.4}.
In order to prove the second assertion we
take $u(N)=[\frac{\log_{2}(\log_{2}(N))}{\la+8}]$ in place of
$[\frac{\log_{2}(N)}{\la+8}]$
and repeat the the proof of the first assertion, with  appropriate
modifications, using the fact that  with $c_{\del}=c^{\del}$,
$n^{2}(\varpi_{q,p}(n)+\be^{\del}(q,n))\leq M\te^{n}$ for
some $1>\te>c_{\del}$ and $\sum_{n=2^{r}}^{\infty}(\be(q,n))^{\del}\leq
r^\del\sum_{n=2^{r}}^{\infty}c_{\del}^{n}=r^\del\frac{c_\del^{2^r}}{1-c_\del}.$
 \qed

\section{Special cases, extensions and concluding remarks}\label{sec6}
\setcounter{equation}{0}

In this section we provide better estimates for the i.i.d. case, extend
results to more general $q_j(n)$ functions and consider also the continuous
time case.

\subsection{Independent case }\label{subsec6.1}
When $\{X(n)\}_{n\geq1}$ are i.i.d. random variables
we do not assume (\ref{2.4}) and (\ref{2.5}) but only that $F(X(1),X(2),...,
X(\ell))$ is a nonconstant random variable having third moment. The
case $\ell=1$ is the "conventional''  case, so we assume that $\ell>1$
and set
\[
S_N=\sum_{n=1}^N F(X(n),X(2n),...,X(n\ell)).
\]
As in Section 2 from \cite{KV2} $S_{N}$ can be splitted  into
sum of independent (blocks) random variables as follows. Let
$l_1,l_2...,l_m\geq 2$ be all primes not exceeding $\ell$. Set
\[
A_n=\{1\leq a\leq n  : a \mbox{ is relatively prime with }  l_1,l_2...
,l_m\}
\]
and
\[
B_s(a)=\{b\leq s:b=al_1^{d_1}\cdot l_2^{d_2}\cdots
l_m^{d_m}\mbox{ for some  nonnegative integers }
d_1,d_2...,d_m\}.
\]
For any $a\in A_N$ put
\[
S_{N,a}=\sum_{b\in B_N(a)}F(X(b),X(2b),...,X(\ell b)).
\]
Then, the distribution of $S_{n,a}$ depends only on $|B_N(a)|$ where $|B|$
denotes here the cardinality of $B$.
Observe that $\{S_{N,a}\}_{a\in A_{N}}$ are independent random variables
and that $S_N=\sum_{n\in A_N}S_{N,n}$.
\begin{assumption}\label{ass6.1}
 Suppose that
\begin{eqnarray*}
&EF(X(1),X(2),...,X(\ell))=0,\,\, 0<d^2=EF^2(X(1),X(2),...
,X(\ell))\\
&\mbox{and}\,\, r^3=E|F^3(X(1),X(2),...,X(\ell))|<\infty.
\end{eqnarray*}
\end{assumption}
\begin{theorem}\label{thm6.2}
Suppose that Assumption 6.1  holds true. Then $\sig^2=\lim_{N\to
\infty}\frac{1}{N}\mbox{Var}(S_N)$ exits and satisfies
\begin{equation}\label{6.0}
\frac 12\big(1-\prod_{k=1}^m(1-\frac 1{l_k})\big)d^2\leq \sig^2\leq\ell^2d^2
\end{equation}
and for all $N\geq 1$,
\begin{equation}\label{6.0+}
d_K(\frac{S_N}{\sqrt N},\cN(0,\sig))\leq C\frac{(1+\log_2 N)^{3m}\max(r^3,1)}
{d\min(d^2,1)\sqrt N}
\end{equation}
 where $C>0$ depends only on $\ell$ (and this dependence can be recovered
 explicitly from the proof below).
\end{theorem}
\begin{proof}
We will use the construction and techniques from Section 4 in \cite{KV2}.
Put $Z_{N,a}=ES_{n,a}^2$, $Z_N=ES_N^2=\sum_{a\in A_n}Z_{N,a}$,
\[
D(\rho)=\{n=(n_1,..,n_m)\in\bbZ^m: n_i\geq0,i=1,..,m\,\, \mbox{ and }
 \sum_{i=1}^m n_{i}\ln(l_i)\leq\rho\}.
\]
Similarly to Section 4 of \cite{KV2} we conclude that $Z_{N,a}$ is determined
only by $|B_N(a)|$ (where $|\Gam|$ for a finite set $\Gam$ denotes its
cardinality), and so we can set $R_l=Z_{N,a}$ if $l=|B_N(a)|$.
 Observe that $R_1=EF^2(X(1),...,X(\ell))=d^2>0$ and that
\begin{equation}\label{6.1}
R_l\leq l^2 EF^2(X(1),...,X(\ell))=l^2d^2
\end{equation}
in view of the inequality $(\sum_{1\leq i\leq l}a_i)^2\leq l\sum_{1\leq i\leq l}
a^2_i$. In Section 4 of \cite{KV2} it was shown that the numbers
\[
\rho_{\max}(l)=\sup\{\rho\geq0:|D(\rho)|=l\}\,\,\mbox{and}\,\, \rho_{\min}(l)=
\inf\{\rho\geq0:|D(\rho)|=l\}
\]
are well defined and satisfy
\begin{equation}\label{6.2}
(l^{\frac{1}{m}}-1)\ln 2<\rho_{\min}(l)<\rho_{\max}(l).
\end{equation}

Next, set
\[
A_N^{(l)}=\{a\in A_N:|B_N(a)|=l\}
\]
and
\[
\hat A_N^{(l)}=\{a\in A_N:Ne^{-\rho_{\max}(l)}\leq a\leq Ne^{-
\rho_{\min}(l)}\}.
\]
In (4.6) and (4.7) from \cite{KV2} it was shown that
\begin{equation}\label{6.3}
|A_N^{(l)}|\leq N2^{-(l^{\frac{1}{m}}-1)}\,\,\mbox{and}\,\,
\frac{1}{N}||A_N^{(l)}|-|\hat A_N^{(l)}||\leq\frac{1}{N}.
\end{equation}
As in (4.10) from \cite{KV2} with $|G_N^{(l)}(n)|=[\frac Nn(e^{-
\rho_{\min}(l)}-e^{-\rho_{\max}(l)})]$,
\[
|\hat A_N^{(l)}|=\sum_{k=1}^m(-1)^{k+1}\sum_{i_1<i_2<...<i_k
\leq m}|G_N^{(l)}(\prod_{s=1}^kl_{i_s})|
\]
and hence
\[
|\frac{1}{N}|\hat A_N^{(l)}|-c_{\ell}(e^{-\rho_{\min}(l)}
-e^{-\rho_{\max}(l)})|\leq\frac{m^m}{N}
\]
where
\[
c_\ell=1-\prod_{k=1}^m(1-\frac{1}{l_k})=\sum_{k=1}^m(-1)^{k+1}
\sum_{i_1<i_2<...<i_k\leq m}\prod_{s=1}^k\frac{1}{l_{i_s}}.
\]
Therefore in view of (\ref{6.3}),
\begin{equation}\label{6.4}
|\frac{1}{N}|A_N^{(l)}|-c_\ell(e^{-\rho_{\min}(l)}-e^{-\rho_{\max}(l)})|
\leq\frac{m^{m}+1}{N}.
\end{equation}
By (4.3) in \cite{KV2},
\begin{equation}\label{6.4+}
|B_N(a)|\leq (1+\frac 1{\ln 2}\ln\frac Na)^m=(1+\log_2\frac Na)^m.
\end{equation}
It follows that
\begin{eqnarray}\label{6.4++}
&\frac1NZ_n=\frac1N\sum_{a\in A_N}Z_{N,a}=\frac1N\sum_{1\leq l
\leq(1+\log_2 N)^m}|A_N^{(l)}|R_l\\
&\underset{N\to\infty}\longrightarrow c_\ell\sum_{l=1}^\infty
(e^{-\rho_{\min}(l)}-e^{-\rho_{\max}(l)})R_l=\sig^2\nonumber
\end{eqnarray}
where the last series converges absolutely in view of (\ref{6.1}) and
(\ref{6.2}).

Next, we will need two following inequalities
\begin{equation}\label{6.5}
|\frac1N\mbox{Var}(S_N)-\sig^{2}|\leq\frac CN(1+\log_2 N)^{3m},
\end{equation}
where $C>0$ depends only on $\ell$, and
\begin{equation}\label{6.6}
c d^2N\leq \mbox{Var}(S_{N})\leq \ell^2 d^2N
\end{equation}
 where $d$ is from Assumption \ref{ass6.1}, $c>0$ depends only on $\ell$
 and we claim the left hand side of (\ref{6.6}) only for
 $N\geq 2\prod_{i=1}^ml_i$.
  Indeed, by (\ref{6.1})-(\ref{6.4}) and by the last equality in (\ref{6.4++}),
\begin{eqnarray}\label{6.6+}
&|\frac1N\sum_{1\leq l\leq(1+\log_2 N)^m}|A_N^{(l)}|R_l
-\sig^2|\\
&\leq d^2\big (\sum_{l>(1+\log_2 N)^m}c_\ell l^2(e^{-\rho_{\min}(l)}
-e^{-\rho_{\max}(l)})+\frac{m^m+1}{N}\sum_{1\leq l\leq(1+
\log_2 N)^m}l^2\big )\nonumber\\
&\leq d^2C\big(\sum_{l>(1+\log_2 N)^m}l^22^{-l^{1/m}}
+\frac{(1+\log_2 N)^{3m}}N\big )\nonumber
\end{eqnarray}
where $C>0$ depends only on $\ell$. Next,
\begin{eqnarray*}
&\sum_{l>(1+\log_2 N)^m}l^22^{-l^{1/m}}=\sum_{j=1}^\infty
\sum_{(j+1)^m(1+\log_2 N)^m\geq l>j^m(1+\log_2 N)^m}
l^22^{-l^{1/m}}\\
&\leq\frac 1N(1+\log_2 N)^{3m}\sum_{j=1}^\infty
(j+1)^{3m}2^{-j}
\end{eqnarray*}
 which together with (\ref{6.6+}) yields (\ref{6.5}).
In order to obtain (\ref{6.6}), observe that $|B_N(a)|=1$ for any
$a\in A_N$ with $a>\frac{N}2$, and so then Var$(S_{N,a})=d^2$. Thus,
\[
\mbox{Var}(S_N)\geq |A_N\cap(\frac N2,N]|d^2.
\]
From the definition of $\rho_{\max}$ and $\rho_{\min}$ it follows that
$\rho_{\max}(l)=\rho_{\min}(l+1)$ which together with (\ref{6.4}) yields
\begin{eqnarray}\label{6.6++}
&\lim_{N\to\infty}\frac{1}N|A_N|=\lim_{N\to\infty}\frac{1}N\sum_{l=1}^\infty
|A_N^{(l)}|\\
&=c_\ell\sum_{l=1}^\infty(e^{-\rho_{\min}(l)}-e^{-\rho_{\max}(l)})
=c_\ell e^{-\rho_{\min}(1)}=c_\ell.\nonumber
\end{eqnarray}
Hence, $\lim_{N\to\infty}\frac{1}N|A_N\cap(\frac N2,N]|=\frac 12c_\ell$
implying the left hand side of (\ref{6.0}).
Observe, in addition, that $|A_N\cap(\frac N2,N]|>0$ whenever $N\geq
2\prod_{i=1}^ml_i$ since then there exists $n\geq N/2$, $n<N$ which is
 divisible by $\prod_{i=1}^ml_i$, and so $a=n+1\leq N$ is relatively
 prime with $l_1,...,l_m$. These yield the left hand side of (\ref{6.6}).
  Now, notice that
 for a given $n\in\bbN$  there are at most $\ell^2$ $m$'s such that
\[
E(F(X(n),...,X(\ell n))F(X(m),...,X(\ell m))\neq 0
\]
while for any $n,m\in\bbN$ the Cauchy- Schwarz inequality implies that
\[
|E(F(X(n),...,X(\ell n))F(X(m),...,X(\ell m))|\leq d^2
\]
and the right hand sides of (\ref{6.6}) and of (\ref{6.0}) follow.

Next, since $|x^{-\frac12}-y^{-\frac12}|=(xy)^{-\frac12}
(x^{\frac12}+y^{\frac12})^{-1}|x-y|$, (\ref{6.5}) and  (\ref{6.6})
 imply that for $N\geq 2\prod_{i=1}^ml_i$,
\[
|\frac1{\sqrt{\mbox{Var}(S_N)}}-\frac1{\sqrt {N\sig^2}}|\leq\frac1{cd^2\sig}
N^{-\frac32}C(1+\log_2 N)^{3m}.
\]
By the left hand side of (\ref{6.6}), $\sig^2\geq cd^2$ which together with
the right hand side of (\ref{6.6}) yields from here that
\begin{equation}\label{6.7}
||\frac{S_N}{\sqrt{N\sig^2}}-\frac{S_N}{\sqrt{\mbox{Var}(S_N)}}||_2\leq
\frac1{c^{3/2}d^2} N^{-1}C\ell (1+\log_2 N)^{3m}.
\end{equation}

Next, using (\ref{6.4+}) and the inequality $(\sum_{i=1}^ka_i)^3\leq k^2
\sum_{i=1}^k|a_i|^3$ we obtain
\[
E|S_{N,a}|^3\leq (1+\log_2 N)^{3m}E|F^3(X(1),...,X(\ell))|=(1+\log_2 N)^{3m}r^3.
\]
Now, in order to prove the last assertion of Theorem \ref{thm6.2} we apply
the assertion 4.1.b from Chapter 4 of \cite{Bai} which yields together
with lower bound from  (\ref{6.6}) that for some absolute constant $C>0$
and all $N\geq 2\prod_{i=1}^ml_i$,
\begin{equation}\label{6.7+}
d_K(\frac{S_N}{\sqrt{\mbox{Var}(S_N)}},\cN(0,1))\leq C(\mbox{Var}
(S_N))^{-
\frac32}\sum_{a\in A_N}E|S_{N,a}|^3\leq\tilde Cr^3\frac{(1+\log_2 N)^{3m}}
{c^{\frac32}N^{\frac12}d^3}
\end{equation}
where $\tilde C>0$ depends only on $\ell$.
 Now, (\ref{6.0+}) follows from (\ref{6.7}), (\ref{6.7+}),  Lemma 3.3
  applied with $a=1$ and  the fact that
\[
d_K(\frac{S_N}{\sqrt N},\cN(0,\sig^2))=d_K(\frac{S_N}{\sig\sqrt N},
\cN(0,1)).
\]
Observe that though we claim (\ref{6.7+}) only for $N\geq 2\prod_{i=1}^ml_i$
the estimate (\ref{6.0+}) holds true for all $N\geq 1$ with some constant $C$
depending only on $\ell$ since by Jensen's inequality $r^2\geq d^2$, and so
 $\max(r^3,1)\geq d\min(1,d^2)$ which enables us to choose $C$ so that
 (\ref{6.0+}) is satisfied also for $1\leq N< 2\prod_{i=1}^ml_i$.
\end{proof}

\subsection{Nonlinear functions $q_{j}$}\label{sec6.2}
Here we discuss the case $k<\ell$, where recall, $q_{j}(n)=jn$ for $j=1,...,k$
and
$q_{j}(n+1)-q_{j}(n)$ and $q_{j}(\ve n)-q_{j-1}(n)$ tend to $\infty$ as $n\to
\infty$ whenever
$\ell\geq j>k$ and $\ve>0$. First, observe that we can exclude the case when
$F(x_1,...,x_\ell)=G(x_1,...,x_k)$ for some Borel function $G$ and $\mu^\ell=
\mu\times\cdots\times\mu$
almost all $(x_1,...,x_\ell)$ since then we arrive at the setup of Theorem
\ref{thm2.4}.
The above equality means that $F$ does not depend essentially on the variables
$x_{k+1},...,x_\ell$
and this is equivalent to the condition that
\begin{equation}\label{6.8}
F_{i}=0\thickspace  \mu^i-\mbox{almost surely (a.s.) for all } i=k+1,...,\ell.
\end{equation}
By Proposition 4.5 in \cite{KV1}, for any $i>k$,
\[
D_{i,i}=\int F_i^2(x_{1},...,x_{i})d\mu^i(x_1,...,x_i),
\]
and so if the above case is excluded then $D_{i,i}>0$ for at least one $i>k$.
This together with Theorem
\ref{thm2.2} yields that then $\sig^2>0$ whence this question is settled here
and it remains to deal only with Berry-Esseen type estimates.

\begin{theorem}\label{thm6.3}
Let $k<\ell$. Suppose that Assumption \ref{ass2.1} is satisfied with some $\al,
\la>1$
and $b\geq4$   and that there exists $1>\gam>0$ such that
$q_{i}([n^\gam])\geq q_{i-1}(n)$ and $q_{i}(n+1)-q_{i}(n)\geq n^{\gam}$  for
any
 $k<i\leq\ell$ and $n\in\bbN$.  Assume that (\ref{6.8}) does not hold true.
Then for any $N\in\bbN$,
\[
d_k(\cL(\xi_N(1)),\cN(0,\sig^2))\leq CRN^{-\frac2{13}\te(\gam,\al,\la)}
\]
where $\te(\gam,\al,\la)=\min(\frac12(1-\gam),\frac{\min(\al,\la)-1}4,\frac{
\gam\min(\al,\la)}{2+\gam\min(\al,\la)}, \frac\la{4(\la+4)})$,  $D_{i,j}$,
$1\leq i,j\leq\ell$ were introduced in Theorem \ref{thm2.2}$,
D_{0,0}=\sig_0=\sum_{1\leq i,j\leq k}\min(i,j)D_{i,j}$,
\[
R=1+(\max_{i\in\{0,k+1,...,\ell\}: D_{i,i}>0}D_{i,i}^{-1})(\max_{k<j\leq\ell\ : D_{j,j}>0}\max(D_{j,j}^
{-\frac43},D_{j,j}^{-\frac45}))
\]
and $C>0$ depends only on the initial parameters and the expressions
(\ref{2.8})-(\ref{2.9}).
\end{theorem}

As in Theorem \ref{thm2.4} the main step in the proof of  Theorem \ref{thm6.3}
 is the construction
of martingale approximations and their estimates. Still,
unlike in the case $k=\ell$ we
cannot provide here approximations of the whole process $\sqrt n\xi_n(1)$ by a
single martingale.
Thus, we will use separately the martingale approximation for $\sqrt n
\sum_{i=1}^k\xi_{i,n}(1)$, $n\geq1$
constructed in Section \ref{sec5} and the martingale approximations of each
$\sqrt n\xi_{i,n}$, $n\geq1$,
$i=k+1,...,\ell$ relying on Lemma \ref{lem6.5} below.

For any $i=k+1,...,\ell$ and fixed $u,N\in\bbN$, we  construct the
 martingales
$(M_{i}^{(u)})_{r}=\sum_{n=1}^rW_{i,q_i(n),2^u}$   with respect to
 the filtration  $\{\cF_{-\infty,q_i(n)+2^u}\}_{n\geq1}$, where similarly
 to Section \ref{sec5},
\[
R_{i,q_i(v),u}=\sum_{s\geq v+1}E[Y_{i,q_i(s),2^u}|\cF_{-\infty,  q_i(v)+2^u}]
\mbox{ and }
\]
\[
W_{i,q_{i}(n),2^{u}}=Y_{i,q_{i}(n),2^{u}}+R_{i,q_{i}(n),u}-R_{i.q_{i}(n-1),u}.
\]

Let $u(N)=[\frac{\log_2(N)}{2(\la+4)}]$.
Using  techniques similar to Section \ref{sec5} we obtain that for any $N>L$
and  $i>k$,
\begin{eqnarray}\label{6.9}
&\frac1{\sqrt{N-L}}||\sum_{n=L}^NY_{i,q_i(n)}-((M_i^{(u(N))})_N
-(M_i^{(u(N))})_L)||_b\leq \\
&C(\frac{2^{u(N)}}{\sqrt{N-L}}+\sum_{r=2^{u(N)}}^\infty
(\be(q,r))^{\del})\leq C_{1}
(\frac{N^{\frac1{2(\la+4)}}}{\sqrt{N-L}}+N^{-\frac\la{2(\la+4)}})\nonumber
\end{eqnarray}
where $C_1>0$ depends only on the initial parameters and the expressions
 (\ref{2.8})--(\ref{2.9}).

Next, observe that the proof of  Lemma \ref{lem5.2} also works  for our setup
and so,
\begin{equation}\label{6.10}
||\sum_{n=1}^{[N^\gam]}Y_{i,q_i(n)}||_{2}\leq CN^{\frac\gam2}
\end{equation}
where $C>0$ depends only on the initial parameters and the expressions
(\ref{2.8})--(\ref{2.9}).
Hence, applying Lemma \ref{lem3.3}  we can replace  $\sum_{n=1}^{N}Y_{i,
q_{i}(n)}$ by $\sum_{n=[N^\gam]}^NY_{i,q_i(n)}$with an error estimated by
(\ref{6.10}). Thus, for $i>k$ and fixed $N$ we
consider the martingales $(\hat M_{i,N})$ (with respect  to the filtration
$\{\cF_{-\infty,q_i(n+[N^\gam])+2^{u(N)}}\}_{n\geq1}$), where  $(\hat
M_{i,N})_r=(M_i^{(u(N))})_{[N^\gam]+r}-(M_i^{(u(N))})_{[N^\gam]}$
for   $N-[N^\gam]\geq r>0$ ,
 $(\hat M_{i,N})_r=\hat M_{i,N-[N^\gam]}$ for $r\geq N-[N^\gam]$.
As in the proof of Theorem \ref{thm2.4} quadratic variation estimates are
crucial. Combining methods of
Proposition 4.5 from \cite{KV1} and Lemmas \ref{lem5.3} and \ref{lem5.4} above
 we obtain the following result.
\begin{lemma}\label{lem6.4}
Suppose that Assumption 2.1 holds true with $\al,\la>1$ and that there
exists $\gam$ satisfying conditions of Theorem \ref{thm6.3}. Let
$k+1\leq i\leq\ell$ and $N,u\in\bbN$ such that  $2^u\leq N^{\frac18}$.
Set $Z_n=Z_n^{(i,u)}=W_{i,q_i(n),2^u}^2$. Then
\begin{equation*}
||\frac1N\sum_{n=1}^NZ_n-D_{i,i}||_2  \leq C(N^{-\tau(\al,\la,\gam)}+2^{-
\frac{\la u}2}+2^{2u}N^{-\frac14})
\end{equation*}
where $\tau(\al,\la,\gam)=\frac12\min(1-\gam,\frac{(\min(\al,\la)-1)}2,
\frac{2\gam\min(\al,\la)}{2+\gam\min(\al,\la)})$
and $C$ depends only on the initial parameters and on (\ref{2.8})-(\ref{2.9}).
Furthermore, let $1\leq i\leq j\leq k$ and set $Z_n=W_{i,n,2^u}W_{i,n,2^u}$.
Then
\begin{equation*}
||\frac1N\sum_{n=1}^{iN}Z_n-iD_{i,j}||_2  \leq C(N^{-\frac{(\min(\al,\la)-1)}4}+2^{-
\frac{\la u}2}+2^{2u}N^{-\frac14}).
\end{equation*}
\end{lemma}

The use of approximations by several martingales as explained above works in
the proof
of Theorem \ref{thm6.3} in view of the following result which is the main
additional argument
needed in comparison to the proof of Theorem \ref{thm2.4}.
Let $g_1(n)<...<g_l(n)$ be positive and
strictly increasing functions taking integer values. Set
\[
K_N=\max_{1<j\leq l}(\min\{1\leq m\leq N:  g_j(m)>g_{j-1}(N)\}).
\]
\begin{lemma}\label{lem6.5}
Let the $\cF_{n,m}$  be a nested family of $\sig-$algebras (see Section
\ref{sec2}).
Let $N\in\bbN$ and suppose that $W^{(i)}$, $i=1,...,l$ is a martingale
 differences sequence with respect to the filtration $\{\cF_{-\infty,
g_i(n+K_{i,N})}\}_{n\geq1}$
where $K_{1,N}=0$ and $K_{i,N}=K_N$ if  $i>1$. Let
$M^{(i)}$  be  the corresponding martingales. Suppose that $\max\{||W_{n}^{(i)}|
|_4
: n\leq N, i\leq l\}\leq C_12^u$ for some positive constant $C_1$
  independent of $N$ and $u\geq0$ such that
$2^u\leq N^\zeta$   for some  $0<\zeta<\frac14$. Let $d_1,...,d_l>0$ and
assume that
\[
\frac1{\sqrt N}\max\{||M_N^{(i)}||_2: i=1,...,l\}\leq C_2
\]
and that $\max\{A_{2,s}: s=1,...,l\}\leq C_3N^{-\te}$ for some $0<\te<1$,
where
\[
A_{2,s}=||\frac1N\sum_{n=1}^N\left(W_n^{(s)}\right)^2-d_s^2||_2
\]
and $C_2,C_3>0$ are  positive constants independent of $N$.
Let $\eta_i,\,i=1,...,l$ be independent  and centered normal random variables having
 variances $d_i^2$. Then
\begin{eqnarray*}
d_K(\frac1{\sqrt N}\sum_{i=1}^l M_N^{(i)},\sum_{i=1}^l
\eta_i)\leq C(1+(\sum_{i=1}^2 d_i^2)^{-\frac12})^{l-1}B
N^{-\frac2{65}\min(5\te,26(1-2\zeta)-8\te)}
\end{eqnarray*}
where $C>0$ is an absolute constant, $B=B(l,C_1,C_2,C_3)=l\max(C_3^\frac25,
l(1+C_2),D\max(C_1^\frac43,C_1^\frac45))$, $D=\max\{D_s:1\leq s\leq l\}$
and $D_s=\max(d_s^{-\frac43},d_s^{-\frac45})$.
\end{lemma}
\begin{proof}
First, observe  that  if $Z_1,Z_2$ and $Z_3,Z_4$ are pairs of independent
 random variables then
\begin{equation}\label{6.11}
d_K(Z_1+Z_2,Z_3+Z_4)\leq d_K(Z_1,Z_3)+d_K(Z_2,Z_4).
\end{equation}
By taking the product measure we can always assume that $\{W_n^{(i)},\,
 n\geq 1\}$ and $\{\eta_i\}_{i=1}^l$ are defined on the same probability
 space and are independent from each other. For any $s=1,...,l$ set
\[
Y^{(s)}=\frac1{\sqrt N}\sum_{i=1}^s M_N^{(i)}\mbox{ and } \del(s)=d_K(Y^{(s)},
\sum_{i=1}^s\eta_i).
\]

The main step of the proof is showing that  for any $2\leq s\leq l$  and $U,L
\in\bbN$
\begin{eqnarray}\label{6.12}
&\hskip0.4cm\del(s)\leq C\big((1+(\sum_{i=1}^s d_i^2)^{-\frac12})(\del(s-1)+D_s(
(\frac{C_1^4N^{4\zeta}}{N^2}UL)^{\frac15}+(\frac{C_1^4N^{4\zeta}}{N^2}UL)^{\frac13})\\
&+(C_3N^{-\theta}UL)^\frac25+l(1+C_2)(\frac1U+\frac1L)^\frac12\big) \nonumber
\end{eqnarray}
where $C>0$ is an absolute constant.

Indeed, for any random variable $Z$ and $U,L\in\bbN$ set
\begin{equation*}
Z_{L,U}=\sum_{k=-LU}^{LU}\frac kU\bbI_{\{{\frac kU\leq Z<\frac{k+1}U\mbox{ and
 }|Z|\leq L}\}}.
\end{equation*}
Observe that by the  H\" older inequality, for any $q>1$,
\begin{equation}\label{6.13}
||Z-Z_{L.U}||_1\leq\frac1U+||E[|Z|\bbI_{\{|Z|>L\}}||_1\leq\frac1U+\frac
{E[|Z|^q]}{L^{q-1}}.
\end{equation}

In order to proceed  we need some relations between probability metrics which
 can be found in \cite{PMT}.
Denote by $d_P$ the Prokhorov metric on $\bbR$ and by $d_L$ the Levi metric
on $\bbR$ (see \cite{PMT}).
Then for any distribution functions $F$ and $G$,
\begin{equation}\label{6.14}
d_L(F,G)\leq d_K(F,G)\leq (1+\sup_{x\in\bbR}|G'(x)|)d_L(F,G)\mbox{ and }
d_L(F,G)\leq d_P(F,G)
\end{equation}
where the right hand side of the first inequality holds true if $G$  is
differentiable.
Moreover, by the  the Markov inequality and some standard estimates, one
can show that
 for any random variables $X$ and $Y$ which are defined on the same probability
space with distribution functions $F$ and $G$,
\begin{equation}\label{6.15}
d_P(F,G)=d_P(X,Y)\leq 2||X-Y||_1^{\frac12}.
\end{equation}

Proceeding with the proof of (\ref{6.12}) we observe that by (\ref{6.14}),
\begin{eqnarray*}
&\del(s)=d_K(Y^{(s-1)}+\frac1{\sqrt N}M_N^{(s)},\sum_{i=1}^s\eta_i)\leq\\
&(1+(\sum_{i=1}^s d_i^2)^{-\frac12})d_L(Y^{(s-1)}+\frac1
{\sqrt N}M_N^{(s)},\sum_{i=1}^s\eta_i).
\end{eqnarray*}
By triangle inequality and then by  (\ref{6.14}) and (\ref{6.15}),
\begin{eqnarray*}
&d_L(Y^{(s-1)}+\frac1{\sqrt N}M_N^{(s)},\sum_{i=1}^s\eta_i)\leq \\
&d_L(Y^{(s-1)}+\frac1{\sqrt N}M_N^{(s)},Y_{L,U}^{(s-1)}+\frac1
{\sqrt N}M_N^{(s)})+d_L(Y_{L,U}^{(s-1)}+\frac1{\sqrt N}M_N^{(s)},Y_{L,U}^{(s-1)}
+\eta_s)\\
&+d_L(Y_{L,U}^{(s-1)}+\eta_s,Y^{(s-1)}+\eta_s)+d_L(Y^{(s-1)}
+\eta_s,\sum_{i=1}^s\eta_i)\leq\\
&4||Y^{(s-1)}-Y_{L,U}^{(s-1)}||_{1}^{\frac12}+d_K(Y_{L,U}^{(s-1)}+\frac1
{\sqrt N}M_N^{(s)},Y_{L,U}^{(s-1)}+\eta_s)+\\
&d_K(Y^{(s-1)}+\eta_s,\sum_{i=1}^s\eta_i)
=I_1+I_2+I_3.
\end{eqnarray*}
By (\ref{6.11}), since $Y^{(s-1)}$ and $\eta_i, i=1,...s$
are independent random variables, $I_3\leq\del(s-1)$.  By  (\ref{6.13})
applied with $q=2$,
\[
I_1\leq4l(1+\frac1{\sqrt N}\max_{1\leq i\leq l}||M_{N}^{(i)}||_2)(\frac1U+
\frac1L)^\frac12\leq4l(1+C_2)(\frac1U+\frac1L)^\frac12.
\]
Next, for any measurable set $A$ satisfying $P(A)>0$ let $P_A=P(\cdot|A)$
be the corresponding conditional probability.
 For any probability measure $\mu$ we denote the expectation with respect to
 it
 by $E_\mu$.
For any $y\in\bbR$ set $A_y=\{Y_{L,U}^{(s-1)}=y\}$ and
$\Gam=\{ y:\, P(A_y)>0\}$. Then, for any $a\in\bbR$ taking into account that
 $\Gam$ is a finite set,
\begin{equation}\label{6.16}
P(Y_{L,U}^{(s-1)}+\frac1{\sqrt N}M_N^{(s)}\leq a)=\sum_{y\in\Gam}P(A_y)P_{A_y}
(\frac1{\sqrt N}M_N^{(s)}\leq a-y).
\end{equation}
If $A\in\cG$ then $E_{P_{A}}[Z|\cG]=\bbI_A E_P[Z|\cG]$
and hence $\{M_r^{(s)}\}_{r\geq 1}$ is also a martingale with respect to the
measure $P_{A_y}$.  Next, we apply (\ref{5.15}) with $\del=1, p=2$ and use
 that $E_{P_A}[|Z|]=\frac1{P(A)}E_P[|Z|\bbI_A]\leq
\frac1{P(A)}E_P[|Z|]$
which yields
\begin{equation}\label{6.17}
|P_{A_y}(\frac1{\sqrt N}M_N^{(s)}\leq a-y)-P(\eta_s\leq a-y)|\leq
AD_s (c^{\frac13}+c^{\frac15}+P^{-\frac25}(A_y)A_{2,s}^{\frac25}),
\end{equation}
where $c=P^{-1}(A_y)C_1^42^{4u}N^{-2}$ and $A>0$ is an absolute
constant. Observe that
 cardinality$(\Gam)\leq3LU$. This together with  (\ref{6.16}) and (\ref{6.17}),
the upper bounds for $A_{2,s}$ and $2^u$
 and the inequality
$\sum_{i=1}^n c_i^t\leq n^{1-t}(\sum _{i=1}^n c_i)^t$ for any $c_i\geq0$ and
$0\leq t\leq1$
yields
\begin{equation*}
I_2\leq A'D_s\big((\frac{C_1^4N^{4\zeta}}{N^2}UL)^{\frac13}+(\frac{C_1^4N^{4\zeta}}
{N^2}UL)^{\frac15}+(C_3ULN^{-\te})^\frac25\big)
\end{equation*}
where $A'$ is an absolute constant and (\ref{6.12}) follows.
Finally, applying (\ref{5.15})  with the martingale $M_N^{(1)}$,
 taking into consideration that $2^u\leq N^{\zeta}\leq\sqrt N$ and
 $A_{2,1}\leq C_3N^{-\te}$ we obtain
\[
\del(1)\leq A\max(C_1^\frac43,C_1^\frac45)D_1
((\frac{N^{4\zeta}}{N^2})^{\frac15}+(C_3N^{-\te})^\frac25).
\]
 Making a repetitive use of (\ref{6.12})
for $s=2,3,...,l$ yields
\begin{eqnarray*}
&\del(l)=d_K(\frac1{\sqrt N}\sum_{i=1}^l M_N^{(i)},\sum_{i=1}^l\eta_i)\leq
Cl(1+(\sum_{i=1}^2d_i^2)^{-\frac12})^{l-1}\big(D\max(C_1^\frac43,C_1^\frac45)\times\\
&((\frac{N^{4\zeta}}{N^2}UL)^{\frac15}+(\frac{N^{4\zeta}}{N^2}UL)^{\frac13})+
(C_3ULN^{-\te})^\frac25+l(1+C_2)(\frac1U+\frac1L)^\frac12\big).
\end{eqnarray*}
The lemma follows by taking  $U=L=[N^{\frac4{13}\te}]$
(the power $\frac{4\te}{13}$ is obtained by considering  $L=U=[N^v]$
and then comparing  the obtained order of $N$ in the last two above summands).
\end{proof}

In order to prove  Theorem \ref{thm6.3}  first apply Lemma \ref{lem3.3} taking
into consideration (\ref{6.9})-(\ref{6.10}). Then apply
Lemmas \ref{lem6.5} and \ref{lem6.4} with the martingales $\hat{M}_{i,N}$ for
$i$'s  such that $D_{i,i}>0$ (where $D_{0,0}=\sig_0$) with $\zeta=\frac1{2(\la+4)}$ and
$\te=\min(\tau(\al,\la,\gam),\frac\zeta2)<2(1-2\zeta)$.

\subsection{Continuous time results}\label{sec6.3}
Here we explain how to obtain convergence rates in the Levy-Prokhorov metric
(\cite{Bil} Ch. 1, Sec. 6) in the case $k=\ell$ for the continuous time
processes $\xi_N(t)$ defined in Section \ref{sec2}. Such
results when $k<\ell$ will not be dealt with here since it is not clear how to
adapt Lemma \ref{lem6.5} for continuous time martingales, and so a different
approach should be employed. It also possible to obtain such rates
for the one dimensional processes $\xi_{i,N}(\cdot)$ for $i=k+1,...,\ell$.
 First, relying on the H\" older inequality for any random variables
 $\{X_i\}_{i\leq n}$,
\[
E[\underset{1\leq i\leq n}{\max}\{|X_i|\}]\leq n^{\frac1q}\max_{1\leq i
\leq n}\{||X_i||_q\}
\]
we obtain by  (\ref{5.3}) that for any $1\leq i\leq\ell$ the martingale
approximation estimates is  in the form
\begin{equation*}
E||M_{i,N}^{(u)}(t)-\xi_{i,N}(t)||_{T,\infty}\leq C(\frac{2^u}{\sqrt N}
(M_i(NT))^{\frac1b}+T\sum_{n=2^{u-1}}^\infty(\be(q,n))^\del)
\end{equation*}
where $||f||_{T,\infty}=\sup\{|f(t)|: t\in[0,T]\}$ and
 $M_{i,N}^{(u)}(t)=\frac1{\sqrt N}\sum_{n=1}^{[Nt]}W_{i,q_i(R_i(n)),2^u}$
with $R_i(n)=\frac ni$ for  $i\leq k$ and $R_i(n)=n$ for $i>k$.
Concerning quadratic variation estimates we obtain the following result.

\begin{lemma}\label{lem6.6}
Suppose that Assumption \ref{ass2.1} is satisfied with some $\al>1,\la>2$
and $b\geq 4$
and that there exists $\gam$ as in Theorem
\ref{thm6.3}. Let $1\leq i,j\leq\ell$ such that $1\leq i\leq j\leq k$
or $i=j>k$ and $u,N\in\bbN$ such that $2^u<N^{\frac1{20}}$. Let $T>0$ and
 $A_T=\sqrt T(T+\frac1T)$. Then
\begin{eqnarray*}
&E||\frac1N\sum_{n=1}^{[Nt]}Z_n-tD_{i,j}||_{T,\infty}\leq  CA_T(2^{-u(\frac
\la2-1)}+
2^uN^{-(\frac{\min(\al,\la)-1}2)}+\\
&2^{2u}N^{-\frac1{10}}+\bbI_{\{i=j>k\}}(2^u N^
{-\min(\frac{\gam\min(\al,\la)}{2+\gam\min(\al,\la)},\frac{1-\gam}5)}))
\end{eqnarray*}
 where $D_{i.j}$ were introduced in Theorem \ref{thm2.2},
 $C$ depends only on the initial parameters and the expressions
(\ref{2.8})-(\ref{2.9}) and for $i,j\leq k$, $Z_n=W_{i,n,2^u}W_{j,n,2^u}$,
while for $i=j>k$, $Z_n=W_{i,q_i(n),2^u}^2$.
\end{lemma}
In order to prove Lemma \ref{lem6.6} we have to improve somewhat Lemmas
\ref{lem5.3} and \ref{lem5.4} obtaining similar results for expressions of
 the form $||\sum_{n=z+1}^N(Z_n-EZ_n)||_2$ with $z<N$ using the same technique
 and then applying Proposition 3 from \cite{MPU}. Finally,  we can apply some
  Berry-Esseen type estimates (for instance, from \cite{Cou}) for continuous
  time martingales  which will yield corresponding estimates in our setup.

\subsection{Integral type expressions }\label{sec6.4}
Next, we discuss how to obtain similar results for expressions of the
form
\begin{equation*}
I_N(t)=\frac1{\sqrt N}\int_0^{Nt}(F(X(q_1(n)),...,X(q_\ell(n)))
-\brF)dt
\end{equation*}
where again $q_i(n)=in$ for all $i$.
 We introduce a reduction to the discrete time case where we can apply the
technique used above for expressions (\ref{1.2}).
First, we represent again the function $F$ in the form (\ref{2.11})
and write
\[
\xi_{i,N}(t)=\frac1{\sqrt N}\int_0^{Nt/i}F_i(X(q_1(t)),...,
X(q_i(t)))dt.
\]

We will use below the same notations as in  $(\ref{3.2})$ with $n$
replaced by $t$ (see Section 6 in \cite{KV1}). In order to apply our
discrete time technique set
$\tilde\xi_{i,N}(t)=\sum_{n=1}^{[Nt/i]}J_i(n)$
where $J_i(n)=\int_0^1Y_i(q_{i}(n+s))ds$. As in $(6.2)$
from \cite{KV1} applied with $\del=b-2$,
\[
P(\sup_{0\leq t\leq T}|\xi_{i,N}(t)-\tilde{\xi}_{i,N}(t)|>\ve)\leq
\frac C{(\ve\sqrt N)^{b-2}}
\]
which by taking $\ve=\ve_{N}=N^{-(\frac12\cdot\frac{b-2}
{b-1})}$
bounds the Levi-Prokhorov  and the Kolmogorov (uniform) distance between
$\tilde\xi_{i,N}$ and $\xi_{i,N}$
by $C\ve_N\leq CN^{-\frac13}$. We can approximate
$\tilde\xi_{i,N}(t)$
by
\[
\tilde\xi_{i,N,r}(t)=\sum_{n=1}^{[Nt/i]}J_{i,r}(n)
\]
where $J_{i,r}(n)=\int_0^1 Y_{i,r}(q_i(n+s))ds$ using an appropriate
version of (\ref{5.3}) (see Section 6 in \cite{KV1}). As
mentioned in \cite{KV1} we will have an appropriate continuous time
version of (\ref{5.2}) with the expressions
\[
R_{i,r}(m)=\sum_{l=m+1}^\infty E(J_{i,r}(l)|\cF_{-\infty,m+r})
\]
and the martingale differences $Z_{i,r}(m)=J_{i,r}(m)+R_{i,r}(m)-R_{i,r}(m-1)$.
In order to extend the results of Section 5 to the present case we should
have a continuous time version of Lemma \ref{lem4.1}. Such a version (adapted
to our specific setup) follows directly from the observation that the
bound from Lemma \ref{lem4.1} depends only on the gaps between the sets
$M_i$ there and the initial parameters together with the facts that for
 $T_j(s_j)=\big(X(q_1(n_j+s_j)),...,X(q_{i_j}(n_j+s_j))\big)$
and integrable functions $G$ and $G_{i}$, $i=1,...,m$,
\begin{eqnarray*}
&E\int_{[0,1]^m}G(T_1(s_1),...,T_k(s_k))ds_1...d_{s_m}\\
&=\int_{[0,1]^m} EG\left(T_1(s_1),...,T_k(s_k)\right)ds_1...
d_{s_m},\\
&E\prod_{j=1}^m\int_0^1 G_i\left(T_j(s_j)\right)ds_j=
E\int_{[0,1]^m}\prod_{j=1}^m G_i\left(T_j(s_j)\right)ds_1...ds_m\,\,
\mbox{and}\\
&\prod_{i=1}^m E[\int_0^1 X_i(s)ds]=E\prod_{i=1}^m\int_0^1 X_i(s)
ds=E\int_{[0,1]^m}\prod_{i=1}^m X_i(s_i)ds_1...ds_m
\end{eqnarray*}
where $X_1(\cdot)$,...,$X_k(\cdot)$ are independent random functions.

Using the technique from Section \ref{sec5} with appropriate modifications
we can prove a corresponding version of Proposition \ref{prop5.1} which
will yield some Berry-Esseen type convergence rates as above.
\qed

\bibliography{matz_nonarticles,matz_articles}

\begin{thebibliography}{Bow75}
\bibliographystyle{alpha}
\itemsep=\smallskipamount

\bibitem{Bai}
Z. Lin and Z. Bai, {\em Probability Inequalities},
Science Press and Springer-Verlag, Beijing and Heidelberg, 2010.

\bibitem{Bil}
P.Billingsley, {\em Convergence of Probability Measures},
2nd ed. Wiley, New York, 1999.

\bibitem{Bow}
R. Bowen, {\em Equilibrium States and the Ergodic Theory of Anosov
Diffeomorphisms}, Lecture Notes in Math. 470, Springer--Verlag, Berlin, 1975.

\bibitem{Br}
R.C. Bradley, {\em Introduction to Strong Mixing Conditions}, Kendrick
Press, Heber City, UT, 2007.


\bibitem{Bro}
F.E. Browder, {\em On the iteration of transformations in noncompact minimal
dynamical systems}, Proc. Amer. Math. Soc. 9 (1958), 773-780.


\bibitem{Cou}
B. Courbot, {\em Rates of convergence in the functional
CLT for multidimensional continuous time martingale}, Stoch.
Proc. Appl. 91 (2001), 57-76.

\bibitem{PMT}
A. L. Gibbs and F. E. Su, {\em On choosing and bounding probability metrics},
Int. Stat. Rev., 70 (2002), 419--435.


\bibitem{Fur}
H.Furstenberg, {\em Nonconventional ergodic averages},
Proc. Symp. Pure Math. 50 (1990), 43-56.

\bibitem{HH}
P.Hall and C.C Heyde, {\em Rates of Convergence in the Martingale Centra
l Limit Theorem},  Ann. Probab. 9 (1981), 395-404.

\bibitem{HK}
Y.Hafouta and Y.Kifer, {\em  A nonconventional local limit theorem},
arXiv: 1407.0143.

\bibitem{IL}
I.A. Ibragimov and Y.V. Linnik, {\em Independent and Stationary Sequences
of Random Variables}, Wolters-Noordhoff, Groningen, 1971.

\bibitem{Ki1}
Yu. Kifer, {\em Nonconventional limit theorems},
Probab. Th. Rel. Fields, 148 (2010), 71-106.

\bibitem{Ki2}
Yu. Kifer, {\em Strong approximations for nonconventional
sums and almost sure limit theorems}, Stoch. Proc. Appl. 123
 (2013), 2286-2302.


\bibitem{KV1}
Yu.Kifer and S.R.S Varadhan, {\em Nonconventional
limit theorems in discrete and continuous time via martingales}, Ann.
Probab. 42 (2014), 649-688.

\bibitem{KV2}
Yu.Kifer and S.R.S Varadhan, {\em Nonconventional large
deviations theorem}, Th. Rel. Fields, 158 (2014), 197-224.

\bibitem{MPU}
F.Merlev`ede, M.Peligrad and S.Utev, {\em Recent
advances in invariance principles for stationary sequences}, Probability
Surveys 3 ( 2006), 1-36.




\bibitem{Ri}
E. Rio, {\em Sur le th\' eor\` eme de Berry-Esseen pour les suites faiblement
d\' pendantes}, Probab. Th. Relat. Fields 104 (1996), 255-282.

\bibitem{RR}
Y. Rinott and V. Rotar, {\em Some bounds on the rate of convergence in the CLT
for martingales}, Theory Probab. Appl. I, 43 (1998), 604-619; II, 44 (1999),
523-536.


\bibitem{Shr}
N. Shiryaev, {\em Probability}, Springer-Verlag, Berlin, 1995.

\end{thebibliography}
\bibliographystyle{alpha}

\end{document}